\documentclass[12pt]{article}
\usepackage[active]{srcltx}
\usepackage[curve]{xypic}
 \usepackage[usenames,dvipsnames]{pstricks} \usepackage{epsfig}
 
\usepackage[active]{srcltx}
\usepackage{epsfig,amsmath}
\usepackage[english]{babel}
\usepackage{amsmath,amsfonts,amssymb,amscd,color,latexsym,epsfig,amsthm}

\newtheorem{newthm}{Theorem}
\newtheorem{theorem}{Theorem}[section]
\newtheorem{lemma}[theorem]{Lemma}
\newtheorem{proposition}[theorem]{Proposition}
\newtheorem{corollary}[theorem]{Corollary}

\newtheorem{definition}[theorem]{Definition}

\def\beginp{ {\noindent\em Proof.} }

\numberwithin{equation}{section}

\newcommand{\norm}[1]{\lVert#1\rVert}

\newcommand{\p}{\partial}
\newcommand{\scal}[1]{{\left\langle#1\right\rangle}}

\def\CCC{{\cal C}}

\def\CCC{{\cal C}}

\def\VVV{{\cal V}}

\def\al{\alpha}

\def\g{\gamma}
\def\G{\Gamma}

\def\ep{\varepsilon}
\def\la{\lambda}

\def\De{\Delta}
\def\de{\delta}

\def\smm{\smallsetminus}

\def\R{\mathbb R}

\def\C{ \mathbb C}

\def\D{\mathbb D}

\def\N{\mathbb N}

\def\lv{ \left(\begin{matrix} }
 \def\rv{\end{matrix}\right)}

\def\cal{\mathcal}

\def\grad{{\rm grad}}

\def\qed{\hfill{q.e.d.}}

\def\ds{\displaystyle}

\newcommand{\mylabel}[1]{\label{#1}}

\newcommand{\REFEQN}[1] { \begin{equation}\mylabel{#1} }
\newcommand{\ENDEQN}{\end{equation}}
\newcommand{\REFTHM}[1] { \begin{theorem}\mylabel{#1} }
\newcommand{\ENDTHM}{\end{theorem}}
\newcommand{\REFNTH}[1] { \begin{newthm}\mylabel{#1} }
\newcommand{\ENDNTH}{\end{newthm}}
\newcommand{\REFPROP}[1]{\begin{proposition}\mylabel{#1} }
\newcommand{\ENDPROP}{\end{proposition} }
\newcommand{\REFLEM}[1]{\begin{lemma}\mylabel{#1} }
\newcommand{\ENDLEM}{\end{lemma} }
\newcommand{\REFCOR}[1]{\begin{corollary}\mylabel{#1} }
\newcommand{\ENDCOR}{\end{corollary} }
\newcommand{\Ref}[1]{ (\ref{#1}) }
\def\beginp{ {\noindent\em Proof.} }
\def\smm{ {\smallsetminus }}
\def\ds{\displaystyle }

\def\mystrut{{\rule[-2ex]{0ex}{4.5ex}{}}}
\def\barz{\overline z}
\def\zbar{\overline z}

\begin{document}
\begin{center}{\large Smooth critical points of planar harmonic mappings \vspace{0.2cm}  \\
M. El Amrani, M. Granger, J.-J. Loeb, L. Tan}\end{center}
 \begin{abstract} In  a work in 1992, Lyzzaik studies local properties of light harmonic
mappings. More precisely, he classifies their critical points and accordingly studies their topological and geometrical
behaviours. We will focus our study on  smooth critical points of  light harmonic maps.
We will establish several relationships between miscellaneous local invariants, and show how to connect them to Lyzzaik's models. 
With a crucial use of Milnor fibration theory, we
get a fundamental and yet quite unexpected relation between three of
the numerical invariants, namely the complex multiplicity, the  local order  of the map and  the  Puiseux pair of the critical value curve.  We also derive 
similar results for a real and complex analytic planar germ at a regular point of its Jacobian level-0 curve. 
 Inspired by Whitney's
   work on cusps and folds,  we develop  an iterative algorithm computing the invariants. Examples are  presented 
 in order to compare the harmonic situation to the real analytic one. 
\end{abstract}
\section{Introduction}

A map $f:W\to \R^2$  defined on a domain $W$ of ${\Bbb R}^2$ is called planar harmonic   
if both components of $f$ are harmonic functions. 
When  $W$ is simply connected, identifying  ${\Bbb R}^2$  with ${\Bbb C}$, 
the map $f$ can also be written in the complex form $p(z)+\overline{q(z)}$ 
where $p$ and $q$ are holomorphic functions on $W$.   We are interested on  germs $f$ 
of planar harmonic maps defined in a neighborhood of a point $a$
and we  will use mainly  the local complex form. 

For such a germ, we restrict ourselves to the situation 
given by the   two following conditions: 
\begin{enumerate}
\item The fiber at $f(a)$ is  the  single point $a$. 
\item The critical set $\CCC_f$  is a  smooth curve  at the point $a$. (The critical set is the vanishing locus 
of  the Jacobian). 
\end{enumerate}
 The first condition  means  that $f$ is light in the sense of Lyzzaik. In our  paper, we will also  present  some remarks for the non light case.  The second condition  implies that $\CCC_f$ is not a single  point. Otherwise, one can prove that 
 up to ${\emph C}^1$ change of coordinates,  $f$  is   holomorphic or 
 anti-holomorphic . 
 
 The critical value set $\VVV_f=f(\CCC_f)$  will play an important  role in our work.

 In our setting, the following natural equivalence relation is introduced: 
 we say that  two germs $f,g$  of planar harmonic maps defined 
 respectively  at points $a$ and $b$, are equivalent 
  if there exists a germ of biholomorphism  $u$ 
 between neighborhoods of $a$ and $b$,  and a real affine  bijection  $\ell$ 
 such that:  $\ell\circ f\circ u= g$.  
For this relation, we get  four numerical invariants and miscellaneous normal forms. 
 These tools allow to understand  analytic and topological facts 
 about germs of harmonic maps. In particular they shed  
 a new light on the geometric  models which appear in the work of  Lyzzaik. 
 
 In order to fully  understand  these invariants, the study of the complexification of $f$ 
 plays a fundamental role. 
 
 Explicitly, the four numerical invariants are: 
\begin{itemize}
\item[--] the absolute value $d$  of the  local topological degree. In our situation, Lyzzaik's work shows that 
  $d$ is $0$ or $1$, 
\item[--]the number $m$  which  is the lowest degree of a non constant monomial 
in the  power series expansion   of the harmonic germ $f$ 
 at the point $a$, 
\item[--] the local  multiplicity $\mu$  of the complexified map of $f$ 
   is defined as  the cardinality of the generic fiber,
\item[--] the number $j$ which is the valuation of the analytic curve $\VVV_f$, in a locally injective parametrization. 
 \end{itemize}
 
 Inspired by Lyzzaik's models \cite{Ly-light}, we prove in a self-contained way 
  that  the germs  $f$ are classified  topologically by
 the numbers $m$ and $d$ or equivalently by $m$ and the parity of $m+j$. 
 We give a simple description of these classes in terms of generalized 
 folds and cusps. 
 
 A main result of our work is that 
the  conditions $j\geq m$ and $\mu= j+m^2$  are  necessary and sufficient 
for the existence of a harmonic planar germ  satisfying the 
conditions 1. and 2. above. 

The  second  relation above was at first  guessed, using many computations. 
Our  proof is based on the theory of Milnor fibration for  germs of holomorphic functions on ${\Bbb C}^2$. 

 The number $j$ occurs also   at  another level. 
In fact we prove   that the critical value set $\VVV_f$ 
can be parametrized by: 
$x(t)= Ct^j +h.o.t$ and $y(t)= C't^{j+1}+h.o.t$, where $C$ and $C'$ are nonzero constants. 
In other words,  $\VVV_f$  is a curve with Puiseux pair $(j, j+1)$,  and   Puiseux theory gives 
a topological characterization of the  complexification of the critical value set. 
This pair  completes also Lyzzaik's description of $\VVV_f$ itself. 

  Every equivalence class of harmonic germs contains   a normal form 
$p(z)^m-\overline{z}^m$, with $p$ a holomorphic  germ tangent  to the 
identity at $0$. When $m=1$,  $\mu$ is equal to the order 
at the origin of the holomorphic germ $\overline{p}\circ p$ 
where $\overline{p}(z)$ is defined as $\overline{p(\overline{z})}$. As a by-product, 
one gets an algorithm to compute $\mu$ in the polynomial case. We obtain also 
relations between the numerical invariants of  $p(z)-\overline{z}$ and 
$p(z)^m-\overline{z}^m$. 
 
 The condition $m=1$ for a harmonic germ   means  that the gradient 
 of its Jacobian does not vanish at the point $a$. In an appendix, we
 generalize results obtained for numerical invariants  in  the harmonic case 
 to  real and complex analytic planar germs satisfying 
 the previous Jacobian condition. In this situation,  we get
 that $\VVV_f$ has still Puiseux pair $(j, j+1)$, with the same geometric consequences. 
 Moreover, as for the harmonic case,  $\mu=j+1$. 
 
 In the analytic case,  we get   also an algorithm  to compute $\mu$ inspired by a 
  fundamental work of Whitney \cite{Wh} on cusps and folds.  At the end of the appendix, some examples are presented 
 in order to compare the harmonic situation to the real analytic one.

 Lyzzaik has also studied the case of  non smooth critical sets. In a forthcoming work, 
 we hope to  extend some of our results to this more general situation. 
 
 {\bf Acknowledgement}. We would like to thank J.H. Hubbard, F. Laudenbach, A. Parusi\'nski and H.H. Rugh for inspiring discussions. 
  
\section{Basic concepts and main results}

Let $K=\R$ or $\C$. Let $U$ be a domain in $K^2$. By convention a domain is a  connected open set. Consider a planar mapping.
$$f: U\to K^2,\quad \lv x\\ y\rv \mapsto \lv f_1(x,y)\\ f_2(x,y)\rv.$$  We say that $f$
is {\it a $K$-analytic map}
if each of $f_1$ and $f_2$ can be expressed locally as convergent power series.
In this case denote by \\
-- $J_f$ the jacobian of $f$\\
--
$\CCC_f= \{ J_f=0\,\}$ the {\it critical set}  \\
-- $\VVV_f=f(\CCC_f)$ the {\it critical value set}.

 We say that \\
-- $z_0\in \CCC_f$ is a {\it regular critical point }of $f$ if $\nabla J_f(z_0)\ne (0,0)$;\\
--
$z_0\in \CCC_f$ is a {\it smooth critical point }of $f$  if $\CCC_f$ is an 1-dimensional   submanifold near $z_0$.

By implicit function theorem a regular critical point is necessarily a smooth critical point. But the converse is not
true. We will see many examples in the following.

We say that $f$ is {\it a planar  harmonic map} if $K=\R$ and each of $f_1$, $f_2$ is $C^2$ with a laplacian equal to zero. Recall that
 $\Delta f_j(x,y)=(\partial^2_x+\partial^2_y)f_j(x,y)$. Note that  in this case each of $f_i$ is locally the real part of a holomorphic map. Thus a planar harmonic map is in particular $\R$-analytic.

The {\it order} of an analytic map $f: K^m\to K^n$ at a point $p$ in the source, is the lowest total degree on a non zero monomial in  the coordinate-wise Taylor expansions of one of the components of $f-f(p)$ around $p$.

For an $\C$-analytic map $F: W\to \C^n$ with $W$ an open set of  $\C^n$ and for a point  
$w_0\in W$, we define
the {\it multiplicity}  of $F$ at $w_0$, by (see \cite{Ch})
$$ \mu(F,w_0)=\limsup_{w \to w_0} \# F^{-1}F(w )\cap X$$
where $X$ is an open neighborhood of $w_0$ relatively compact in $W$
such that $F^{-1}F(w_0)\cap \overline X=\{w_0\}$. If such $X$ does not exist, set $\mu(F,w_0)=\infty$.

In the situation above there is an open neighborhood $U$ of $F(w_0)$ and an open dense subset $U_1\subset U$ such that
for all $w\in U_1$, $\mu(F,w_0)=\# F^{-1}F(w )\cap X$. 
For a holomorphic map $\eta: U\to \C$ with $U$ an open set of  $\C$, and $w_0\in U$, the two notions coincide.

For a $\R$-analytic map $f: U\to \R^2$, in particular a planar harmonic map,  we define its multiplicity at a point to be the multiplicity of its holomorphic extension in $\C^2$.  We will also frequently use the well known fact that for a holomorphic map $f: U\to \C^2$ its multiplicity at a point $p_0=(x_0,y_0)\in U$ is equal to the codimension in the ring of power series of the ideal defined  by its component : \[
\mu(f,p_0)=\dim_{\C}\frac{\C\{u,v\}}{f_1(x_0+u,y_0+v),f_2(x_0+u,y_0+v)}.
\]
 See for example  \cite[theorem 6.1.4]{dJP}.

One objective of this work is to show that  harmonic maps around a smooth critical point of a given order have only two types
of topological behaviours, depending on the parity of the multiplicity.

Our investigation is based on Whitney's singularity theory on $C^\infty$ planar mappings, multiplicity theory of holomorphic maps of two variables and Lyzzaik's work on light harmonic mappings.

A smooth 1-dimensional manifold in $\R^2$ admits a smooth parametrization. If the critical set of a harmonic mapping is smooth somewhere, there is actually a parametrization that is in some sense natural. This induces a natural parametrization $\beta(t)$ of the critical value set.

\begin{definition} We denote by $R_{\ge k} (s)$ a convergent power series on $s$ whose lowest power in $s$ is at least $k$. A planar 
$K$-analytic curve $\beta: \{|s|<\ep\}\ni s\mapsto \beta(s)\in K^2 $ is said to have the \emph{order pair   $(j,k)$} at $ \beta(0)$, for some $1\le j<k\le \infty$, if up to reparametrization in the source and an analytic change of coordinates in the range $K^2$, the curve takes the form
$\beta(s)=\beta(0)+ \lv \mystrut Cs^j + R_{\ge j+1} (s)\\ C's^k + R_{\ge k+1 }(s)\rv$ with $C\cdot C'\ne 0$.
\end{definition}

Let us assume that $k$ is not a multiple of  $j$, and that the complexified parametrization of $\beta$ is locally injective. This is the case in particular if $(j,k)$ are co-prime. 
Then the order $j$ and more generally the order-pair $(j,k)\in \N^2$ is an analytic invariant of the curve independently of such a  parametrisation  : $j$ is the minimum, and $k$ the maximum of the intersection multiplicities $(\beta,\gamma)$ among all smooth $K$-analytic curves $\gamma$. This order-pair is also a topological invariant of the complexified curve because $\frac1{gcd(j,k)}(j,k)$ is its first Puiseux pair. Such a unique order-pair exists unless $j=1$, or $j=+\infty$.

Our main goal is to establish a relationship between the order of the critical value curve and the multiplicity, and then to connect these invariants to Lyzzaik's topological models. More precisely, we will prove:

\REFTHM{main} Let $f$ be a planar  harmonic map in a neighborhood of  $z_0$  with $z_0$ as a smooth critical point. 
\begin{enumerate} 
\item (Critical value order-pair)    The critical value curve has a natural parametrization and an order $j$  at $f(z_0)$. It has an order-pair of the form $(1,\infty)$ if $j=1$, and  $(j,j+1)$ if $1<j<\infty$. 
\item (Critical value order and multiplicity)  The three invariants $m$ order of $f$, $j$ order of the critical values curve and $\mu $ multiplicity of the complexified map on $(\C^2,z_0)$ are related by the  (in)equalities : 
\REFEQN{triple} \left\{\begin{array}{l} \infty\ge j\ge m\ge 1\\
 j+m^2=\mu.\end{array}\right.
\ENDEQN
\item (Topological model)  
Assume $\mu<\infty$. Let $\D$ be the unit disc in $\C$. There is a neighborhood
$\De$ of $z_0$, a pair of orientation preserving homeomorphisms
$h_1: \De\to \D, \ z_0\mapsto 0, \quad h_2 : \C\to \C,\ f(z_0)\mapsto 0$, 
and a pair of positive odd integers $2n^\pm-1$ satisfying \Ref{N-plus} below, such that
$$h_2\circ f\circ h_1^{-1}(re^{i\theta})=\left\{\begin{array}{ll} re^{i(2n^+-1)\theta} & 0\le \theta\le \pi\\
re^{-i(2n^--1)\theta} & \pi\le  \theta\le 2\pi\ .
\end{array}\right.$$ 
Moreover  $\# f^{-1}(z)=n^++n^-$ or 
$n^+ + n^- -2$ depending on whether $z$ is in one sector or the other of $f(\De)\smm \beta$.
\REFEQN{N-plus}\text{\fbox{$ \begin{array}{cl} \text{ $\mu$ even},&  \lv \mystrut 2n^+-1\\ 2n^--1\rv  \in \left\{ \lv m \\ m \rv, \lv m+1 \\ m+1 \rv\right\} \\
 \text{$\mu$ odd}, &  \lv \mystrut 2n^+-1\\ 2n^--1\rv  \in  \left\{ \lv m+1 \\ m-1 \rv, \lv m-1 \\ m+1 \rv, \lv m \\ m+2 \rv, \lv m+2 \\ m \rv\right\} \end{array}$}}
\ENDEQN\end{enumerate}
\ENDTHM 

We want to emphasise that guessing and proving the relation $\mu=j+m^2$ was the main point of the work. Our starting point was the case $m=1$, which corresponds to the critically regular case. We could establish then $\mu=j+1$. But this case does not indicate a general formula. A considerable amount of numerical experiments have been necessary to reveal a plausible general relation, and then results in singularity theory about Milnor's fibres had to be employed to actually prove the relation.
 
We will deduce the topological model  from our formula \Ref{triple} and a result of Lyzzaik \cite{Ly-light}.  More precisely, we will parametrize the critical value curve $\beta$ in a natural way and then express its derivative, as did Lyzzaik, in the form
 $\beta'(t)=Ce^{it/2}\cdot R(t)$, with $C\ne 0$ and $R(t)$ a real valued analytic function. 
 
 Lyzzaik defined in his Definition 2.2 the singularity to be of the {\em first kind} if $R(t)$ changes signs at $0$, which is equivalent to $j>0$ even, and of the {\em second kind} if $R(0)=0$ and $R(t)$ does not change sign at $0$,  which is equivalent to $j\geq1$ odd. He then deduced the local geometric shape of $\beta$ (cusp or convex) in Theorem 2.3  following the kind. What we do here is to push further his calculation to determine the order-pair of the critical value curve $\beta$, which then gives automatically its shape (cusp or convex).
 
Lyzzaik then provided  topological models  in his  Theorem 5.1  following the parity of an integer $\ell$ (which corresponds to our $m-1$) and the kind (or the shape of $\beta$) of the
singularity, corresponding in our setting to the parity of $m+j$. Thanks to our relation \Ref{triple},
we may then express Lyzzaik's topological model in terms of the parity of $\mu$.

Lyzzaik's  proof   relies on previous results of Y. Abu Muhanna and A. Lyzzaik \cite{AL}. We will reestablish his models
with  a self-contained proof.

As a side product, we obtain the following existence result (which was a priori not obvious):
 
\REFCOR{existence} Given any triple of integers
$(m,j,\mu)$ satisfying \Ref{triple}, there is a harmonic map $g(z)$ with a
smooth critical point $z_0$ such that $Ord_{z_0}(g) =m$, $\mu(g,z_0)=\mu$, and 
$(j, j+1)$ is the order-pair of the critical value curve at $g(z_0)$. 

Given any pair of  integers $n^\pm\ge 1$ satisfying $\left\{\begin{array}{ll} n^+=n^-\text{\ \ or}\\
|n^+-n^-|= 1\end{array}\right.$,  there are two consecutive integers $k,k+1$ and harmonic maps
with order $m=k$ and $m=k+1$ respectively  realizing the topological model \Ref{N-plus} for the pair $n^\pm$ and the order $m$. \ENDCOR

\section{Normal forms for planar harmonic mappings}

Recall that any real harmonic function on a simply connected domain
in $\C$ is the real part of some holomorphic function. Therefore, if $U
\subset \C$ is simply connected, and $f\,:\,U \to \C$ is a harmonic
mapping, then  $f=p+\overline{q}$ where $p$ and $q$ are holomorphic
functions in $U$ that are unique up to additive constants. We will say that $p+\overline{q}$ is a {\it local expression} of $f$. In a study around a point $z_{0}$ we will often take
the unique local expression in the form
$\,f(z)=f(z_{0})+p(z)+\overline{q(z)}$ with $p(z_{0})=q(z_{0})=0$.

\subsection{Existence and unicity of the normal forms}

\begin{definition}{\emph  A natural equivalence relation}. For $Z,W$ open sets in $\C$, with $z_0\in Z$, $w_0\in W$, and for
 harmonic mappings $f:Z\to \C$ and $g:W\to \C$, we say that $(f,z_0)$ and $(g,w_0)$ are {\it equivalent} and we write
 $$(f,z_0)\sim (g,w_0)$$
  if there is a bijective $\R$-affine map
$H: \C\mapsto \C, z\mapsto az+b\zbar+c$ and a biholomorphic map $h:W'\to Z'$ with $z_0\in Z'\subset Z$, $w_0\in W'\subset W$  such that 
$h(w_0)=z_0$ and $g=H\circ f\circ h$ on $W'$.
\end{definition}

\REFLEM{local normal forms} Let $f$ be a non-constant harmonic map
 defined on a neighborhood of $z_0$.  Then \REFEQN{Local} (f,z_0)\sim (g,0) \quad \text{for some}\quad g(z)=z^m-\overline{z^n(1+O(z))}\ENDEQN with $\infty\ge n\ge m\ge 1$ (here $O(z)$ denotes a holomorphic map  near $0$ vanishing at $0$). 
 
 Moreover if another map $G(z)=z^M-\overline{z^N(1+O(z))}$ with $\infty\ge N\ge M\ge 1$ satisfies
 $(G,0)\sim (g,0)$ then $(M,N)=(m,n)$. If $m<n$ then  $g(z)=\dfrac1{c^m}G(cz)$ for $c$ an $(m+n)$-th root of unity. 
   \ENDLEM
\beginp  We may assume $z_0=0$ and $f(0)=0$.

I. We may assume that  $f$ is harmonic on a simply connected open neighborhood  $V$ of $0$. One can thus write $f(z)=p(z)-\overline{q(z)}$ with $p,q$ holomorphic on $V$. Replacing $f$ by $f-f(0)$ we may assume 
$f(0)=0$, and we may also assume $p(0)=q(0)=0$. 

Case 0.  Assume $p\equiv 0$ or $q\equiv 0$. Replacing $f(z)$ by $\overline{f(z)}$ if necessary we may assume $q\equiv 0$.
In this case $p(z)=az^m(1+O(z))$ with $a\ne 0$  and there is a bi-holomorphic map $h$ so that $p(z)=(h(z))^m$. Therefore $f(z)=g(h(z))$ with
$g(w)=w^m$.

II. Assume now that none of $p,q$ is a constant function.  Replacing $f(z)$ by $\overline{f(z)}$ if necessary we may assume 
 $p(z)=az^m(1+O(z))$ and $q(z)=bz^n(1+O(z))$ with $\infty> n\ge m\ge 1$ and $a\cdot b\ne 0$.
 
Replacing $f$ by $(\overline{b\la})^{-n}f(\la  z)$  changes $a$ to $(\overline{b\la})^{-n}a\cdot \la ^m$ and $b$ to $1$. We may thus assume $f(z)=Az^m(1+O(z))-\overline{z^n(1+O(z))}$,  $A\ne 0$.
 
 Case 1.  $m< n$. Choose $\rho$ so that $\dfrac {A\cdot \rho^m}{\overline \rho^n}=1$.
 Replace $f$ by $\dfrac{f(\rho z)}{\overline \rho^n}$ we may assume 
 $f(z)=z^m(1+O(z))-\overline{z^n(1+O(z))}$.

 Case 2. $m=n$. We choose $\tau$ so that $\dfrac {A\cdot \tau^m}{\overline \tau^n}=\dfrac {A\cdot \tau^m}{\overline \tau^m}\in \R_+^*$.
 We may thus assume $$f(z)=cz^m(1+ O(z))-\overline{ z^m(1+O(z))},\quad c>0.$$
 If $c=1$ we stop. Assume $c\ne 1$.  Then $H(z):=z+\dfrac1c \zbar$ is an invertible linear map.
 And as $c$ is real, we get easily :
 \[
  H(f(z))= (c-\dfrac1c)z^m(1+O(z))-\overline{O(z^{m+1})}\qquad \text{with } c-\dfrac1c\ne 0\ .\]
 Replacing $f$ by $H\circ f$ we are reduced to Case 0 or Case 1.
 
 Therefore in any case we may assume $$f(z)=z^m(1+O(z))-\overline{z^n(1+O(z))},\quad 1\le m\le n, \ m<\infty,\ n\le \infty.$$
 Now there is a holomorphic map $h$ with $h(0)=0$, $h'(0)=1$ defined in a neighborhood of $0$ so that the holomorphic part of $f$
 can be expressed as $h(z)^m$. Then
 $$f\circ h^{-1}(z)=z^m-\overline{z^n(1+O(z))},\quad 1\le m\le n, \ m<\infty,\ n\le \infty$$ on some neighborhood of $0$. This establishes the existence of 
 normal forms.
 
 Let us now take a map $G(z)=z^M-\overline{z^N(1+O(z))}$ with $\infty\ge N\ge M\ge 1$ so that
 $(G,0)\sim (g,0)$ with $g(z)=z^m-\overline{z^n(1+O(z))}$ and $m\le n$. It is easy to see that $M=m$.
 Let now $h(z)=cz(1+O(z))$ be a  holomorphic map with $c\ne 0$ and $H(z)=az+b\overline z$ so that $H\circ G\circ h(z)=g(z)$.
 Then $$a\cdot (h(z))^m -a\cdot \overline{( h(z))^N(1+ O(z))}+ b\cdot \overline{(h(z))^m} -b \cdot ( h(z))^N(1+ O(z))=z^m -\overline{z^n(1+ O(z))}. $$
 Assume $n>m=M$. If $N=m$ then the terms $z^m$ and $\zbar^m$ have coefficients $(a-b)c^m$ and $(b-a)\bar c^m$ on the left hand side, and $(1,0)$ on the right hand side. This is impossible. So $N>m$ as well.
 
 Comparing   the  $\overline z^m$ term on both sides we get $b=0$, and then the $z^m$ term we get 
 $ac^m=1$.
 Comparing then the holomorphic part of both sides we get $h(z)=cz$. Now the anti-holomorphic part gives $N=n$ and $\overline a c^n=1$.
 It follows that $\overline c^{-m} c^n=1$. So $|c|=1$, $c^{m+n}=1$ and $c^{-m}G(cz)=g(z)$.
 \qed
 
 We remark that in the case $m=n$ the normal form is not unique. Here is an example:
 
 Let $G(z)=z+ iz^2-\barz$. For any $\Re b\ne -\frac12$  the map is equivalent to $G(z)+bG(z)+b\overline{G(z)}=(z+iz^2+bi z^2)-\overline{z+\overline b i z^2}=w+ O(w^2) -\bar w=: g(w)$
 for $w=z+\overline b i z^2$. 
 
 \subsection{Criterion and normal forms for critically smooth points}

We say that a  subset set $Q$ of $\C$ is a {\bf locally regular star} at $z_0$ of $\ell$-arcs if there is a neighborhood $U$ of $z_0$
and a univalent holomorphic map $\phi: U\to \C$ with $\phi(z_0)=0$ so that $Q\cap U=\{z, \phi(z)^\ell\in \R\}$.
If $\ell=1$ then $Q$ is a smooth arc in $U$.

 \REFLEM{Smooth} Let $f$ be a harmonic map in a neighborhood of $z_0$. The following conditions are equivalent:
\begin{enumerate}\item[1)]  $\CCC_f$ is a non-constant smooth $\R$-analytic curve in a neighborhood of $z_0$.
\item[2)]   For $m:=Ord_{z_0}(f)$,  in a local expression $f(z)=p(z)+\overline{q(z)}$, we have $m =Ord_{z_0}(p)=Ord_{z_0}(q)<\infty$, the map $\psi(z):=\dfrac{p'(z)}{q'(z)}$ extends to a holomorphic map at $z_0$, with $|\psi(z_0)|=1$ and $\psi'(z_0)\ne 0$.
\item[3)] $(f, z_0)\sim (g,0)$ with \REFEQN{degenerate} g(z)=z^{m}+bz^{m+1}+O(z^{m+2}) +\overline{z^{m}}, \quad  |b|=1;\ENDEQN
\end{enumerate}
Every equivalence class  of such $(f,z_0)$ has a representative in any of the following forms (with any choice of signs): $ (f, z_0)\sim (h,0) $ with 
\REFEQN{plus minus}\  h(z)=\pm z^{m}+bz^{m+1}+O(z^{m+2}) \pm \overline{z^{m}}  \text{\quad or\quad } h(z)=\pm \Big(z+bz^2+O(z^{2})\Big)^m \pm \overline{z^{m}}, \quad  |b|=1.\\
\ENDEQN Furthermore, $z_0$ is a regular critical point if and only if $m=1$.
\ENDLEM 

\beginp Assume at first $f(z)=p(z)+\overline{q(z)}$, with
 $$p(z)=z^m+ bz^{m+k} + O(z^{m+k+1}),\quad q(z)=z^m,\quad k\ge  1, \ b\ne 0.$$ Note that $J_f=|p'|^2-|q'|^2$.
Set $\psi(z)=\dfrac{p'(z)}{q'(z)}$. We have
 $$\CCC_f=\{J_f=0\}=\{q'=0\}\cup \{ |\psi |=1\}=\{0\}\cup \psi^{-1}(S^1)=\psi^{-1}(S^1) \ .$$
But  $\psi^{-1}(S^1)$ is a locally regular star at $0$ of $k$-arcs. So $\CCC_f$ is smooth at $0$ if and only if $k=1$, or equivalently, $\psi'(0)\ne 0$.
 
 This proves in particular the implication 3)$\Longrightarrow$1).
 
 Let us prove 1)$\Longrightarrow$3).
We may assume $f$ is in the local normal form \Ref{Local}. 
If $m\ne n$ then it is easy to see that $z_0$ is an isolated point of $\CCC_f$.  This will not happen under the smoothness assumption of $\CCC_f$.
So $m=n$. 

Replace $f$ by $\overline{f(az)/a^m}$ with $a^{2m}=-1$ we have $f(z)=p(z)+\overline{q(z)}$, with
 $$p(z)=z^m+ bz^{m+k} + O(z^{m+k+1}),\quad q(z)=z^m,\quad k\ge  1, \ b\ne 0.$$  Since $\CCC_f$ is smooth at $0$ by the argument above we have $k=1$.
  We may then replace $f$ by $f(\la z)/\la^m$ for $\la=\dfrac1{|b|}>0$ to get a normal form so that $|b|=1$. 
This is \Ref{degenerate}.

The rest of the proof is similar. We leave the details to the reader.
\qed

\section{Order $j(f,z_0)$ of the critical value curve for a harmonic map}

For a harmonic map near a smooth critical point, we will introduce what we call the natural parametrization of the critical value curve, and then compute its order-pair in this coordinate.

 Points 3, 4 and 5 of the following result are due to Lyzzaik, \cite{Ly-light}.
Just to be self-contained we reproduce Lyzzaik's proof here (with a somewhat different presentation).

\REFLEM{curves}      Assume $f(z)=p(z)+\overline{q(z)}$ is an harmonic mapping and is critically smooth at $z_0$. Set $\psi(z)=\dfrac{p'(z)}{q'(z)}$ and $m=Ord_{z_0}f$. 
\begin{enumerate}\item We have $\la:=\psi(z_0)\in S^1$ and $\psi'(z_0)\ne 0$. 
 The critical set $\CCC_f$  in a neighborhood of $z_0$ coincides with $\psi^{-1}(S^1)$, is  locally a smooth arc. We endow this arc what we call the natural parametrization by  
$\g(t):=\psi^{-1}(\la e^{it})$;\item We then endow the
 critical value set what we call its natural parametrization by  $\beta(t):=f(\g(t))$.  Set $j=Ord_0(\beta(t))$. Either
 $\beta(t)\equiv \beta(0)=f(z_0)$,  in which case $j=+\infty$ by convention, or  $\infty> j\ge m $.
\item For the line $L=\{f(z_0)+ s\sqrt\la,s\in \R\}$, the set $f^{-1}(L)$ is a locally regular star at $z_0$ with $2(m+1)$ branches. 
\item We have  
$\beta'(t)=\sqrt{\la e^{it}}R(t)$, with $R(t)=2\Re \Big( \sqrt{\la e^{it}} \dfrac d{dt}q(\g(t))\Big)$, an  $\R$-analytic real function of $t$. 
\item In the case $\beta'(t)\not\equiv 0$, the curve $t\mapsto \beta(t)$  is locally injective, has a strictly positive curvature in a punctured neighborhood of $0$,  turns always to the left,  is tangent to $L$ at $\beta(0)$.
 \item We have  $j-1=Ord_0(R)$ and  $\infty\ge j\ge m $.  Either $j=\infty$ and $\beta\equiv \beta(0)$,
 or  the curve $\beta$ has the order-pair    $(j,j+1)$ at $0$. \end{enumerate}
 \ENDLEM
 
 \beginp Point 1.  We have
 $$\CCC_f=\{J_f=0\}=\{q'=0\}\cup \{ |\psi |=1\}=\{q'=0\}\cup \psi^{-1}(S^1)\ .$$
 But $q(z)$ is not constant (otherwise $\psi\equiv \infty$) we know that $\{q'=0\}$ is discrete and avoids a punctured neighborhood
 of $z_0$. Therefore, reducing $U$ if necessary, we have
 $\{J_f=0\}\cap U=\psi^{-1}(S^1)\cap U$, and we may
 choose a holomorphic branch of $\sqrt{\psi(z)}$ for $z\in U$.
From Lemma \ref{Smooth} we know that $\psi(z_0)\in S^1$, $\psi'(z_0)\ne 0$ and so $\psi$ is locally injective. Reducing $U$ further if necessary, we see that
 $\{J_f=0\}\cap U$ is a smooth arc. We call the parametrization $\g(t)=\psi^{-1}(\psi(z_0)\cdot e^{it})$ the natural parametrization of $\CCC_f$.
 
 Point 2. We endow the critical value set with the natural image parametrization $\beta(t)=f(\g(t))$.
 
 Write $f(z)=p(z)+\overline{q(z)}$, $p(z)=a(z-z_0)^m+h.o.t.$ and $q(z)=q(z_0)+A(z-z_0)^{m}+ h.o.t.$ for some $a,A\ne 0$. Due to the smoothness of the critical set at $z_0$, we have 
 $\g(t)= z_0+\g'(0)\cdot t+ h.o.t.$ with $\g'(0)\ne 0$. So
   $$\beta(t)=f(\g(t))=p(\g(t))+\overline{q(\g(t))}=\beta(0)+(a \g'(0)^m+ \overline A \overline{\g'(0)^m}) t^m + h.o.t\ .$$
   It follows that $j=Ord_{0}\beta(t)$ satisfies $m\le j\le +\infty$. 

Point 3. Without loss of generality we may assume $z_0=0$ and $f(z_0)=0$. Choose a local expression $f(z)=p(z)+\overline{q(z)}$
so that $p(0)=q(0)=0$. Then $p(z)=az^m + O(z^{m+1})$ and $q(z)=bz^m + O(z^{m+1})$ for some $a,b\ne 0$. 
We have $\dfrac ab= \psi_f(0)=\la$. Rewrite now $f$ in the form
$$f(z)=\sqrt{\la}\Big(P(z) + \overline{Q(z)}\Big) = \sqrt{\la} \Big(P(z)-Q(z) + Q(z)+\overline{Q(z)}\Big)$$
with $P(z)=p(z)/\sqrt \la$. Then $P(z)$ and $Q(z)$ have  identical coefficient for the term $z^m$ and  $Ord_0(P)=Ord_0(Q)=m$. 
Set  \REFEQN{F} F(z)=P(z)-Q(z),\ r(z)=Q(z)+\overline{Q(z)}\quad \text{so that}\quad f(z)=\sqrt\la \Big(F(z)+ r(z)\Big).\ENDEQN
Note that $r(z)$ is real-valued, and $F(z)$ is holomorphic with multiplicity greater than $m$.
Write $F$ in the form $F(z)=c z^{m+n}(1+O(z))$ with $c\ne 0$ and $n\ge 1$. 
As $P(z)=Q(z)+ F(z)$, we have $$\psi_f(z)=\dfrac {\sqrt\la P'(z)}{\overline{\sqrt\la}Q'(z)}=\la\Big( 1+\dfrac{F'(z)}{Q'(z)}\Big) .$$
It follows that $n=Ord_0\psi_f$. But $Ord_0\psi_f=1$ by the smoothness assumption of  the critical set. So $n=1$
and $F$ takes the form $F(z)=c z^{m+1}(1+O(z))$ with $c\ne 0$.

Finally $f^{-1}L=\{f(z)\in \sqrt \la\cdot \R\}=\{F(z)+r(z)\in \R\}=\{F(z)\in \R\}=F^{-1}\R$. This set is therefore a locally regular star of
$2(m+1)$ branches.

Point 4. We follow the calculation of Lyzzaik. Let $z\in \CCC_f$. Then $|\psi(z)|=1$. So there are two choices of  $\sqrt{\psi(z)}$.
Fix a choice of the square root.

 $$\begin{array}{rcl}Df|_z&=&p'(z)dz+\overline{q'(z) }d\barz \\
&=&\mystrut q'(z)\psi(z)dz +\overline{q'(z) } d\barz\\
&=& \sqrt{\psi(z)}\left(\sqrt{\psi(z)}q'(z)dz+ \overline{ \sqrt{\psi(z)}q'(z)}d\barz \right)\\
&=&\mystrut   \sqrt{\psi(z)}\,\Re\, \left(2\sqrt{\psi(z)}q'(z) dz \right).\end{array}$$
As  $\g(t)$ is defined by $\psi(\g(t))=\la e^{it}$, we have
$$\beta'(t)=Df|_{\g(t)}(\g'(t))=   \sqrt{\la e^{it}}\, R(t),\quad \text{where}\quad R(t)= \Re\, \left(2\sqrt{\la e^{it}}\dfrac d{dt}q(\g(t)) \right).$$

Points 5 and 6. Assume that $\beta$ is not constant. Then  $Ord_0(\beta)=j<\infty$, $R(t)\not\equiv 0$ and $Ord_0R=j-1$. So $\dfrac{\beta'(t)}{R(t)}\to \sqrt\la$ as $t\to 0$.
It follows that $\beta(t)$ is tangent to $L$ at $\beta(0)$. Furthermore, 
$$R(t)=C(t^{j-1}+bt^j+O(t^{j+1}))\ , \ C\in \R^*,\ b\in \R\ .$$
A simple calculation shows that $\beta''(t)=\left(\dfrac{R'(t)}{R(t)} + \dfrac i2\right)\beta'(t)$. As $\dfrac{R'(t)}{R(t)}$ is real,
we see already that the oriented angle from $\beta'$ to $\beta''$ is in $]0, \pi[$. One
can also check the sign of  the curvature of $\beta$: 
\REFEQN{Curvature} \kappa_\beta(t)=\dfrac{\Im\,(\overline{\beta'(t)} \cdot \beta''(t))  }{|\beta'(t)|^3}=\dfrac 1{2|\beta'(t)|} > 0,\quad t\in ]-\de,\de[\smm \{0\}\ .\ENDEQN
This shows that there is some $\de>0$ such that $\beta(t)$ is on the left of its tangent for any $t\in ]-\de,\de[\smm\{0\}$ if $\beta'(0)=0$ and for any $t\in ]-\de, \de[$ if $\beta'(0)\ne 0$.

Moreover,
$$\beta(t)=\beta(0)+ \int_0^t \beta'(s) ds =\beta(0)+  C\sqrt{\la} \int_0^t e^{is/2}\left(s^{j-1}+b s^{j}+ h.o.t.\right)ds $$
$$=\beta(0)+  C\sqrt{\la}  \int_0^t  \left(s^{j-1}+ (b+\dfrac i 2) s^{j}+ h.o.t. \right)ds.$$
So
$$\Re \dfrac{\beta(t)-\beta(0)}{C\sqrt{\la}}=\ds\int_0^t  s^{j-1}(1+ O( s^{j}) )ds,\quad \Im \dfrac{\beta(t)-\beta(0)}{C\sqrt{\la}}=\ds \int_0^t  \dfrac{s^{j}}2(1+ O( s^{j+1}) )ds.$$ It follows that $t\mapsto \beta(t)$ is locally injective and $\beta$ has the order pair $(j,j+1)$ at $0$.
\qed

\begin{definition}  Let $f$ be a harmonic map and $z_0$ be a smooth critical point. 
We denote by $j(f,z_0)$ the integer so that the critical value curve has the order-pair $(j(f,z_0), j(f,z_0)+1)$ in its natural parametrization.
We will call $j(f,z_0)$ the {\it critical value order} of $f$ at $z_0$. 

Let us notice that $j(f,z_0)$ is an analytic invariant hence is a fortiori invariant under our equivalence relation on harmonic maps.
\end{definition}

\section{Between critical value order and multiplicity}
The objective here is to prove the following 

\REFTHM{general}  Given a harmonic map $G$ together with a smooth critical point $z_0$, the three local analytic invariants $m$ order of $f$, $j$ order of the critical values curve and $\mu $ multiplicity of the complexified map on $(\C^2,z_0)$ are related by the  (in)equalities : 
$$ \left\{\begin{array}{l} \infty\ge j\ge m\ge 1\\
 j+m^2=\mu.\end{array}\right.
$$
\ENDTHM

\subsection{A  formula for the multiplicity $\mu$}

For $p(z)=\sum a_iz^i$ we use $\overline p(z)$ to denote the power series $\overline p (z)= \sum \overline a_i z^i$.  The following lemma provides a formula for the multiplicity,  which in the case of a polynomial $p$ leads to an algorithm.

\REFLEM{first} Let $p(z)$ be a holomorphic map with $p(0)=0$. Let $f(z)=p(z)-\bar z$ and  $g(z)=p(z)^m-\zbar^m$  (with $m\ge 1$ an integer).   Then $\mu(f,0)=Ord_0(\bar p \circ p(z)-z)$ and more generally :
$$ \mu(g,0)= \sum_{\xi^m=\eta^m=1}Ord_0 \Big(\eta\, \overline p (\xi \, p (z))-z\Big). $$
\ENDLEM
 \beginp Consider the following holomorphic extensions of $f$ and $g$ in $\C^2$:
$$M_f: \lv u\\ v\rv \mapsto \lv p(u)-v\\ \overline p(v)-u\rv,\quad M_g:   \lv u\\ v\rv \mapsto \lv p(u)^m-v^m\\ (\overline p(v))^m-u^m \rv.$$
By definition $\mu(f,0)=\mu(M_f, {\bf 0})$ and $\mu(g,0)=\mu(M_g,0)$. Let us work directly with $M_g$.  It is known for example by \cite[theorem 6.1.4]{dJP} that $\mu(g,0)<\infty$ if and only if the germs of planar curves $p(u)^m-v^m=0$ and $p(v))^m-u^m$ have no branch in common. This condition means that there is a neighborhood of $(0,0)$ in $\C^2$, in which $\lv 0\\  0\rv $ is the only solution of the  system of equations $M_f\lv u \\ v\rv =\lv 0\\  0\rv $. Since this system is equivalent to the existence of $\xi,\eta$, such that 
\[
\xi^m=\eta^m=1,\text{ and } v=\xi p(u), \quad \eta\overline{p}(\xi p(u))-u=0
\]
the condition $\mu(g,0)=\infty$ is indeed equivalent to the finiteness of the order in the right-hand side of the statement of lemma \ref{first}.

We denote $\mu_1=\sum_{\xi^m=\eta^m=1}Ord_0 \Big(\eta\, \overline p (\xi \, p (z))-z\Big)$ this order.
Solving the equation $M_g\lv u \\ v\rv =\lv 0\\  t\rv $, we get 
\begin{equation*}
  \left\{
      \begin{aligned}
 v-\xi p(u)=0  \\
 \prod_{\xi^m=\eta^m=1}(\eta\overline{p}(\xi p(u))-u)=t      
   \end{aligned}
    \right.
\end{equation*}

There are $\mu_1$ distinct solutions in the variable $u$ for the second equation, hence $ \mu_1$ solutions for the system 
which merge at a single solution $(0,0)$ when $t\to 0$.
These solutions are all simple which means that $M_g$ is locally invertible. Applying again \cite[theorem 6.1.4]{dJP}, this proves that $\mu(M_g,{\bf 0})=\mu_1$. 
\qed

Note that  $|p'(0)|\ne 1$ iff $\mu(f,0)=1$. Otherwise $\mu(f,0)\ge 2$.

\subsection{Normalizations}

\REFLEM{relation} Any harmonic map $G$ near a smooth critical point $z_0$ is equivalent to $(g, 0)$ with $g(z)=p(z)^m-\barz^m$
for some integer
 $m\ge 1$  and some holomorphic function $p(z)=z + b z^2 + O(z^3)$, $|b|=1$. Furthermore, setting $f_\xi(z)=\xi \cdot p(z)-\barz$, $\xi\in \C$, we have
 \[
 \mu(G,z_0)=\mu(g,0)=\left\{\begin{array}{ll} m^2+m & \rm{ if \;} (-b^2)^m\ne 1 \\
\mu(f_{-b^2},0)+ (m-1)(m+2)> m^2+m& \text{\rm otherwise.} \end{array} \right. 
\]
\ENDLEM
Note that in the particular case $m=1$, the above formula becomes
$$\mu(G,z_0)=\mu(g,0)=\left\{\begin{array}{ll} 2 & \text{if } b^2\ne -1 \\
\mu(f_1,0)> 2& \text{otherwise.} \end{array} \right. $$

\beginp  The existence of the model map $g$ follows  from Lemma \ref{Smooth}. In the following the sums are over the $m$-th roots of unity for both $\eta$ and $\xi$. By Lemma \ref{first}, 
\begin{eqnarray*}\mu(g ,0) &=& \sum_{\xi^m=\eta^m=1}Ord_0 \Big(\eta\, \overline p (\xi \, p (z))-z\Big)\\
&=&  \left[\sum_{\eta\xi\ne 1}+  \sum_{ \eta\xi=1 }\right]Ord_0 \Big(\eta\, \overline p (\xi \, p (z))-z\Big) \\
&=& \sum_{\eta\xi\ne 1} Ord_0 \Big(\eta\, \overline p (\xi \, p (z))-z\Big)  +\sum_{ \xi^m=1}Ord_0 \Big( \overline{\xi p} (\xi \, p (z))-z\Big)\\
&\overset{Lem. \ref{first}}=& \sum_{\eta\xi\ne 1} Ord_0 \Big(\eta\, \overline p (\xi \, p (z))-z\Big)  +\sum_{ \xi^m=1,\xi\ne -b^2} Ord_0\Big( \overline{\xi p} (\xi \, p (z))-z\Big) + C\mu(f_{-b^2},0)\end{eqnarray*}
where $C=0$ if $-b^2$ does not coincide with any $m$-th root of unity, and $C=1$ otherwise.

There are $m(m-1)$ pairs of $(\eta,\xi)$ with $\eta^m=1=\xi^m$ and
$\eta\xi\ne 1$. For each pair of them,
$Ord_0 \Big(\eta\, \overline p (\xi \, p (z))-z\Big)=1$. This gives $m(m-1)$ for the first sum above.

Now for any $\xi$ with $\xi^m=1,\xi\ne -b^2$, we have $Ord_0\Big( \overline{\xi p} (\xi \, p (z))-z\Big)=2$. 
If $-b^2$ does not equal to any $m$-th root of unity, there are $m$ terms in the middle sum above, so $\mu(g,0)=m(m-1)+2m=m^2+m$. 
Otherwise there are $m-1$ terms, so $\mu(g,0)=m(m-1)+2(m-1) + \mu(f_{-b^2},0)=m^2+m-2 + \mu(f_{-b^2},0) $.
In this case one can check easily that $ \mu(f_{-b^2},0)>2$. So $\mu(g,0)> m^2+m$. 
\qed

Consider now
 $$g(z)=p(z)^m-\barz^m=\Big(z + b z^2 + O(z^3)\Big)^m-\barz^m, \quad |b|=1 .$$
A direct calculation using the first term of $\g(t)=\psi_g^{-1}(-e^{it})$ shows that the critical value curve $\beta$  in its natural parametrization satisfies 
$\beta(t) = \dfrac {2i}{(m+1)^m} \Im\Big(\dfrac{i^m}{b^m} \Big)t^m + o(t^m)$.

Clearly $m=Ord_0(g)$. Let $j$ be the order of $\beta(t)$ at $0$, and $\mu$  the multiplicity of $g$ at $0$. We want to prove 
$$j\ge m\quad  \text{and\quad} \mu= j+m^2.$$

Note that for $|b|=1$, $$(-b^2)^m=1 \Longleftrightarrow \Big(\dfrac i{ \overline{b}}\Big)^{2m}=1  \Longleftrightarrow \Big(\dfrac i b\Big)^{2m}=1
 \Longleftrightarrow  \Big(\dfrac i b\Big)^m=\pm 1  \Longleftrightarrow  \Im\Big(\dfrac{i^m}{b^m}\Big) =0\ .$$

This, together with Lemma \ref{relation}, gives:

\REFCOR{generic case} (The generic case) For $p(z)=z + b z^2 + O(z^3)$, $|b|=1$ with $(-b^2)^m\ne 1$,
and $g(z)=p(z)^m-\barz^m$, we have $$j =m\text{\quad and \quad}\mu= j + m^2=m+m^2.$$
If  $(-b^2)^m= 1$  then $j > m$. \ENDCOR

It remains to work on the degenerate case $(-b^2)^m=1$.

\REFLEM{b=i} Any harmonic map of the form $g(z)=(z+bz^2+o(z^3))^m-\barz^m$ with $(-b^2)^{m}=1$ is equivalent to a map of the form $(z+iz^2+o(z^2))^m-\barz^m$.\ENDLEM
\beginp One just need to replace $g$ by $g(\la z)/\overline{\la}^m$ for  $\la=1/(-ib)$. \qed

\subsection{The normalised degenerate case}

The following statement will complete the proof of Theorem \ref{general}. This is by far the hardest case.

\REFTHM{key-relation}  Let  $p(z)=z + i z^2 + O(z^3)$ be a holomorphic map in a neighborhood of $0$ and  $m\ge 1$ be an integer. Set $g(z)=p(z)^m-\barz^m$.  
Then $g$ is a harmonic map with $0$ as a smooth critical point. Let $j$ be the order of the critical value curve at $g(0)=0$ in its natural parametrization, and $\mu$  the multiplicity of $g$ at $0$. Then 
$$ j>m\quad\text{and}\quad \mu= j+m^2.$$
\ENDTHM
\beginp We know already that $0$ is a smooth critical point of $g$ and $j>m$ (Corollary \ref{generic case}).
Let's look at the complexification of $g$: 

$G\lv u \\ v\rv=\lv p(u)^m-v^m\\ -u^m+\overline p (v)^m \rv=\lv u^m(1+iu+o(u))^m - v^m \\
-u^m+ v^m(1-iv+o(v))^m \rv= \lv G_1\\ G_2\rv.$

The critical set in $\C^2$ of $G$ contains the set  $\{(uv)^{m-1}=0\}$ which consists of two branches $u=0$ and $v=0$.

The corresponding critical value branches are 

$G\lv 0 \\ v\rv = \lv -v^m\\  v^m(1-iv + O(v^2))^m \rv ,\  G\lv u \\ 0\rv =\lv u^m(1+i u+O(u^2))^m\\ -u^m\rv.$ 

Both are plane curves with order pair $(m, m+1)$. The other branch gives a critical value curve $\beta$ with order pair $(j, j+1)$, as we already know from the real calculation. By comparing the two parametrizations, an elementary calculation shows that these two branches are distinct. They are also distinct from the third branch since we shall prove that $j>m$ hence that they have different first Puiseux pairs. 

The local behavior of $G$ at each of these critical branches, off the origin, is given by the following:

\REFLEM{local} The multiplicity of $G$ at a real critical branch point (off the origin) is 2, and the
multiplicity of $G$ at a non-real critical branch point (off the origin) is $m$. 
\ENDLEM 
\beginp
The expression of $G$ at the point $\lv 0 \\ v_0\rv$ in local coordinates $\lv u \\ w\rv=\lv u \\ v-v_0\rv$ is 
$$G\lv u \\ v_0+w\rv-\lv -v_0^m\\ \overline p(v_0)^m\rv  =\lv p(u)^m-mv_0^{m-1}w+O(w^2))\\
-u^m+Q(v_0)w+O(w^2)\rv .
$$ 
By using to Taylor formula for $\overline{p}(v_0+w)^m$ we find $Q(v_0)=m\overline{p}(v_0)^{m-1}\overline{p}'(v_0)$. In order to see that the germ of $G$ at the point $\lv 0 \\ v_0\rv$ is equivalent by analytic coordinates changes to the germ $\lv \hat{u} \\ \hat{v}\rv \to \lv \hat{u}^m \\ \hat{v}\rv$, it is sufficient to check that $Q(v_0)\neq  mv_0^{m-1}$ for any small enough non zero $v_0$. The proof for the branch $u\to G\lv u \\ 0\rv$ is similar.
\qed

The preimage $G^{-1}(S)$, of $S$ a small
sphere centered at the origin, is a smooth 3-manifold, and in fact we are going to prove a stronger result stating that the pair $(G^{-1}(B),G^{-1}(S))$ is diffeomorphic to the pair made of the standard ball and the standard sphere.
	
	We notice that $G^{-1}(S)$ is defined by the equation $N(u,v):=\left\Vert F\lv u \\ v\rv\right\Vert^2= \epsilon^2 $, and is the boundary of $G^{-1}(B)=\left\{\lv u \\ v\rv\mid\left\Vert F\lv u \\ v\rv\right\Vert^2\leq \epsilon^2\right\}$.  The above result will then follow  from a general statement about a function $N$ defined on an open set of $\R^n$ given in the next lemma.

\begin{lemma}\label{isotopy}
Let $N : W\to \R$ be a positive real analytic map defined in a neighborhood of $0\in \R^n$ such that $N^{-1}(0)=\{0\}$. Then there is $\epsilon_0 >0$ such that for $0<\epsilon\leq \epsilon_0$ the pair of sets 
\[
\left(\{x\in \R^n\mid N(x)\leq \epsilon^2\},\{x\in \R^n\mid N(x)= \epsilon^2\}\right)
\] is diffeomorphic to the standard ball and the standard sphere.
\end{lemma}
\beginp
First we prove that $N$ is a submersion outside the origin if we restrict to a small enough neighborhood of $0$ :  
there is a constant $r_0>0$ such that if $0<\norm{x}\leq r_0$ we have 
\[
\grad N(x):=\left(\frac{\p N}{\p x_1},\dots,\frac{\p N}{\p x_n}\right)(x)\neq0.
\]
Indeed if this was not true, we could by the curve selection lemma \cite[lemma 3.1]{Mi} find an analytic path $\gamma : [0,\eta_0[\longrightarrow W$ such that $ \gamma(0)=0$ and $\gamma(t)\neq0$ for $t\in ]0,\eta_0[$, and $\grad N(\gamma(t))=0$. But then we would have $\frac{d}{dt}(N(\gamma(t)))=\scal{\gamma'(t),\grad N(\gamma(t))}=0$. But then $N(\gamma(t))$ would be constant equal to $N(\gamma(0))=0$ and this contradicts $\gamma(t)\neq0$ for $t\neq0$. 

In a second step we show that the gradient of $N$ tends to point out from $0$ when $t\to 0$.
More precisely this means that given an analytic path $\gamma : [0,\eta_0[$ such that $ \gamma(0)=0$ and $\gamma(t)\neq0$ for $t\in ]0,\eta_0[\to W$, we have:   
\[
\underset{t\to 0}{\lim\;\;}\frac{\scal{\gamma(t),\grad N(\gamma(t))}}{\norm{\gamma(t)}\cdot\norm{\grad N(\gamma(t))}}\geq0.
\]
Indeed let $\alpha,\beta$ be the valuations of $\gamma$ and $\grad N\circ\gamma$. We have power series expansions with initial vector coefficients $a,b\in \R^4$ : 
\[
\gamma(t)=at^\alpha+o(t^\alpha), \quad \grad N(\gamma(t))=bt^\beta+o(t^\beta)
\]
and the limit above is  $\frac{\scal{a,b}}{\norm{a}\norm{b}}.$ The expansion of the derivative of $\gamma$ is $\gamma'(t)=\alpha at^{\alpha-1}+o(t^{\alpha-1})$, and therefore $\frac{d}{dt}(N(\gamma(t)))=\scal{\gamma'(t),\grad N(\gamma(t))}=\alpha\scal{a,b}t^{\alpha+\beta-1}+o(t^{\alpha+\beta-1})$. Since $N(\gamma(t))>0$ for small enough positive $t$, this forces the inequality  $\scal{a,b}\geq 0$ and we are done. 

We deduce a quantified version of this behaviour of the gradient vector field, showing that the angle of the vectors $x,\grad N(x)$ is bounded away from $\pi$.
Precisely making the constant $r_0$ above smaller if necessary
we may assume that for $0<\norm{x}\leq r_0$ :
\[
\frac{\scal{x,\grad N(x)}}{\norm{x}\cdot\norm{\grad N(x)}}\geq -\frac12.
\]
This claim is a consequence of the curve selection lemma, because the limit property of $\grad N(x)$ implies that $0$ cannot be in the closure of the semi analytic set 
\[
Z:=\{x\in W \mid 0<\norm{x}\leq r_0, \quad \scal{x,\grad N(x)}<-\frac12\norm{x}\cdot\norm{\grad N(x)}\}.
\]
Our third and last step is to show that we have a homotopy between $\Sigma =N^{-1}(\epsilon ^2)$ and the standard ball $\norm{x}^2=\epsilon ^2$ because the gradient of interpolations between $N$ and $\norm{x}^2$ never vanishes outside the origin. Indeed the choice we made for $r_0$ has the following consequence: for any $t\in [0,1]$, we have $2tx+(1-t)\grad N(x)\neq 0$ and this implies that the the relative gradient with respect to $(x_1,\dots, x_n)$ of the deformation $N(t,x):=t\norm{x}^2+(1-t)N(x)$ is non zero for any $x\neq 0$:
\[
\forall t\in [0,1],\forall x, 0<\norm{x}\leq r_0 , \quad \grad_xN(t,x)=\left(\frac{\p N}{\p x_1},\dots,\frac{\p N}{\p x_n}\right)(t,x)\neq0 .
\]
Using the continuity of $N$ let us choose $\epsilon_0<r_0$ such that $N(x)\leq \epsilon _0^2\Longrightarrow \norm{x}<r_0$. Then the property we obtained on the gradient shows that for each $t\in [0,1]$ the set $\Sigma_t:=N_t^{-1}(\epsilon^2))$ (resp $\Sigma:=N^{-1}(\epsilon^2))$ is a submanifold of the open ball $B(0,r_0)$ (resp. of the product $[0,1]\times B(0,r_0)$). The set $B_t=N_t^{-1}([0,\epsilon_0^2])$ is a manifold with boundary $S_t$ and interior an open set of $\R^n$. Similarly $\Sigma$ is a part of the boundary of $B=N^{-1}([0,\epsilon_0])\subset [0,1]\times B(0,r_0)$ to be completed by $B_0\cup B_1$\footnote{we might avoid easily to consider a manifold with a corner along $S_0\cup S_1$ by enlarging slightly the range of $t$ to an open interval $]-\eta,1+\eta[$.}. We notice that $(B_1,\Sigma_1)$ is the standard ball of radius $\epsilon_0$ with its boundary.

Finally the restriction to $\Sigma $ of the projection $(t,x)\longrightarrow t$ is a submersion. This implies by the version with boundary of a well known theorem of Ehresmann \cite{Eh} that the pair $(B,\Sigma)$ is locally trivial above $[0,1]$
which means that we have a diffeomorphism 
\[
(B,\Sigma)\longrightarrow [0,1]\times (B_0,\Sigma_0).
\]
In particular we have a diffeomorphism $(B_0,\Sigma_0)\longrightarrow  (B_1,\Sigma_1)$ as expected.
\qed

Let us now come back to the map $G$. Take a small  round closed ball $D$ of radius $\epsilon_0$ and its preimage $B$ so that $G:B\to D$ is a covering of degree $\mu$ outside the
critical value curves, and that $\partial D$ is transverse to the critical value set. 
It follows from lemma \ref{isotopy} applied to $N=\norm{G}^2$ that $\partial B$ is a smooth 3-variety diffeomorphic to a sphere. We take $r_0$ as in this lemma and denote : $B_\epsilon=\{x\mid N(x)\leq \epsilon^2\}\subset D_{r_0}$, $\Sigma_\epsilon=\p B_\epsilon$ for all $0<\epsilon\leq\epsilon_0$.

 Let $\ell: \lv x_1\\ x_2\rv\mapsto ax_1 + bx_2$ be a generic linear form. For $\epsilon_0$ small enough the disc  $(\ell=0)\cap D$ intersects the critical value set  $\VVV$ only at the origin.
Therefore for $t$ a small enough non zero complex number, the line $L_t$ with equation 
 $\ell(u,v)=t$ is transversal to the boundary of $D$ and $L_t\cap D$ is  a disc  $\Delta_t$. Furthermore if $t\neq0$ $L_t$, intersects the critical value set  $\VVV$ at $j+2m$ points contained in the interior of $D$. 
 
 Set $Y_t :=\{\ell(G_1,G_2)=t\}=G^{-1}(\{\ell=t\})$. 
 \begin{proposition}
With well chosen $r_0,\epsilon_0$, as in the proof of lemma \ref{isotopy} and $t\neq0$ small enough, $X_t:=Y_t\cap B_{\epsilon_0}$ is diffeomorphic to the Milnor fiber of the function $\ell(G_1,G_2)$. 
 \end{proposition} 
\begin{proof} In the proof of lemma \ref{isotopy} we may choose if necessary a smaller $r_0$ to guarantee that the standard ball $B'_{r_0}=\Big\{\lv u \\ v\rv\mid \Big\| \lv u \\ v\rv\Big\|\leq r_0\Big\}$ is a Milnor ball which means 
that $X_0$ is transverse to the standard sphere $\p B'_{r}$ for each $r\in ]0,r_0]$ and the Milnor  fiber is by definition $X_t\cap B'_{r_0}$ for $0<|t|\leq \eta_0$, with $\eta_0$ small enough. By this very definition $B'_r$ is also a Milnor ball and $X_t\cap B'_r$ a Milnor fiber provided that we restrict the condition on $t$ to $0<|t|\leq \eta$ for an appropriate $\eta<\eta_0$. In fact for such a $t$ the inclusion $X_t\cap B'_r\subset X_t\cap B'_{r_0}$ yields a deformation retract between two diffeomorphic varieties. Now we have the inclusion $B_{\epsilon_0}\subset B'_{r_0}$ and choosing $r$ small enough to get $B'_r\subset B_{\epsilon_0}$ we can perform again the construction of lemma \ref{isotopy} and we get the chain of inclusions: 
\begin{equation}\label{1}
B_{\epsilon}\subset B'_r\subset B_{\epsilon_0}\subset B'_{r_0}.
\end{equation}
Let us choose $\eta _0$ small 
enough both for the validity of the Milnor fibration and for the transversality of the intersections  $L_t\cap \p D$ as described above, with $D$ of radius ${\epsilon_0}$. We have to notice also that $L_0$ is transverse to $D_{\epsilon}$ for all $\epsilon\leq \epsilon _0$. Then at any point $y\in \p X_t=Y_t\cap \Sigma_{\epsilon_0}$, the two varieties $Y_t$ and $\Sigma_{\epsilon_0}$ are also transversal, and so are $Y_0$ and $\Sigma_{\epsilon}$ for $0<\epsilon\leq \epsilon_0$. Indeed at such a point $y$ we have avoided $\VVV$ and the map $G$ is a local diffeomorphism. 

Because of these transversalities we can construct Milnor fibrations with Milnor fiber $Y_t\cap B_\epsilon$ using "pseudo Milnor balls" $B_\epsilon$ which make a basis of neighborhoods of $0$. The arguments are exactly the same as with the standard Milnor fibration. It is known (see \cite{Le}, Theorem 3.3)  that this Milnor fiber is diffeomorphic to the standard one. The proof uses the chain of inclusions (\ref{1}). Indeed we choose $t$ small enough for the intersections of $Y_t$ with the four terms in  (\ref{1}), to be Milnor fibers. The two inclusions $Y_t\cap B_\epsilon\subset Y_t\cap B_{\epsilon_0}$ and $Y_t\cap D_r \subset Y_t\cap D_{r_0}$ are homotopy equivalences. Therefore, in the sequence of maps 
\[
\xymatrix{H^i(Y_t\cap B_\epsilon)\ar[r]^{\alpha_1}&H^i(Y_t\cap D_r)\ar[r]^{\alpha_2}&H^i(Y_t\cap B_{\epsilon_0})\ar[r]^{\alpha_3}&H^i(Y_t\cap D_{r_0})}
\]
$\alpha_3\circ \alpha_2$ and $\alpha_2\circ \alpha_1$ are isomorphisms and this forces the middle arrow to be an isomorphism for $i=0,1$. Since we work on surfaces with boundaries this is enough to obtain that they are diffeomorphic.
\qed 
\end{proof}
\begin{proposition} The surface $X_t$ is connected and $\chi(X_t)= 2m-m^2$. Furthermore its boundary has $m$ connected components and its genus is $g(X_t)=\frac{(m-1)(m-2)}2$.
\end{proposition} 
 Proof. Since $X_t$ is a smooth real surface with boundary its Euler characteristic is $\chi(X_t)=1-\dim(H^1(X_t),\C)$ because it is connected by \cite{Mi}. The first statement in the proposition is equivalent to the fact that the Milnor number $\mu(\ell\circ G)=\dim(H^1(X_t),\C)$ $(m-1)^2$. To check this fact recall that $\mu(\ell\circ G)$ is an analytic invariant (and even a topological one) of the function.
  Let us calculate a standard form up to an analytic change of coordinates, for $L:=\ell(G_1,G_2)$:
\begin{eqnarray*}
 &L(u,v)&=a(p(u)-v^m)+b(-u^m+\overline p (v)\\
& &=(a-b)u^m(1+O(u))- (a-b)v^m(1+O(v))=U^m-V^m\\
 \end{eqnarray*} where $\Phi:\lv u\\ v\rv \mapsto \lv \varphi(u)\\ \psi(v)\rv $ is a diagonal change of coordinates. We can now check that $\mu(\ell(G_1,G_2))=(m-1)^2$ by the formula for the Milnor number as the codimension of the Jacobian ideal : $\mu(L)=\dim_\C\C\{u,v\}/(\frac{\p L}{\p u},\frac{\p L}{\p v})=\dim_\C\C\{U,V\}/(U^{m-1},V^{m-1})$.
 The last statement follows since the number of components of the boundary is the number of irreducible local components of the curve $L(u,v)=0$.
 \qed
 
Now we are ready to finish the proof of theorem \ref{key-relation}. We already know that $F: X_t \to \Delta_t$ is a ramified cover of degree $\mu$ with $j+2m$ critical values and that above each critical value there is exactly one  critical point. 

 By the proof of  \ref{local} we know that the germ of the map $F$, at a critical point different from $(0,0)$, is up to analytic changes of coordinates,  equivalent to one of the two germs $(z_1,z_2)\to (z_1,z_2^2)$ or $(z_1,z_2)\to (z_1,z_2^m)$. 
Since the disc $\Delta_t$ is transversal to the critical value curve, we deduce  that for $F: X_t \to \Delta_t$  the critical points  are simple on the smooth branch, and of local multiplicity $m$ (therefore counts as $m-1$ critical points), above the fantom curves.

By Riemann-Hurwitz, $\chi(X_t)+\#\{\text{critical points}\}=\mu\chi(\Delta)=\mu$.
So $1-(m-1)^2+(j+2m(m-1) )=\mu$.  That is $\mu=j+m^2$. \qed

Combining with Lemma \ref{relation}, in which we plug in $b=i$, $\mu(f_{-b^2},0)=\mu(f,0)$ we get:
\REFCOR{with regular} For $m\ge 1$,  $f(z)=(z+iz^2+O(z^3))-\zbar$ and $g(z)=(z+iz^2+O(z^3))^m-\zbar^m$, the four quantities $j(f),\mu(f), j(g), \mu(g)$ at $0$ are related as follows:
$$\mu(g)=\mu(f)+m^2+m-2 ,\quad  j(g)=\mu(g)-m^2=j(f)+m-1=\mu(f)+m-2.$$
In particular each of these number determines the three other ones.
\ENDCOR

\section{Topological models for harmonic smooth critical points}
 
Notice that due to the equality $\mu=j+m^2$, the integers $\mu$ and $m+j$ have the same parity. 
In this section we will reformulate Lyzzaik's topological model in terms of the parity of $m+j$. We provide a self-contained proof.

We then show examples of harmonic maps with prescribed numerical invariants or with prescribed local models.

\subsection{Local models}

\REFTHM{Lyzzaik2}(topological model, inspired by Lyzzaik, \cite{Ly-light}) Let $f$ be a harmonic map
with $z_0$ a smooth critical point. Set $m=Ord_{z_0}(f)$. Let $j$ be the integer so that the critical 
value curve $\beta$ at $z_0$ has the order-pair    $(j,j+1)$. Assume $j<\infty$.

In this case, define $n^\pm$ by the following table:
\REFEQN{n-plus}\begin{array}{|c|| c|c|}\hline  \lv \mystrut 2n^+-1\\ 2n^--1\rv  &\mystrut  \text{ $\beta$ convex ($m\le j$ odd)} &  \text{$\beta$ cusp ($m\le j$ even)} \\ \hline\hline
m\text{ odd}  & \lv m \\ m \rv & \lv m+2 \\ m \rv \quad \text{or}\quad {\lv m  \\ m+2 \rv}  \\ \hline
m\text{ even}  & \lv m+1 \\ m-1 \rv \quad \text{or}\quad {\lv m-1  \\ m+1 \rv} & \lv m+1 \\ m+1 \rv 
\\ \hline \end{array}\ENDEQN

Set $R_{n^+, n^-}(z)= R_{n^+, n^-}(re^{i\theta})=\left\{\begin{array}{ll} re^{i(2n^+-1)\theta} & 0\le \theta\le \pi\\
re^{-i(2n^- -1)\theta} & \pi\le  \theta\le 2\pi\ .
\end{array}\right.$

Then  one of the choices of $R_{n^+, n^-}(z)$ (the choice is unique if $m+j$ is odd) is a local topological model of $f$, in the following sense: There is a neighborhood
$U$ of $0$, two orientation preserving homeomorphisms:
$h_1: U\to \D, \ 0\mapsto 0, \quad h_2 : \C\to \C,\ 0\mapsto 0$, 
 such that
$$h_2\circ f\circ h_1^{-1}(z)=R_{n^+, n^-}(z).$$ 
Moreover  $\# f^{-1}(z)=n^++n^-$ or 
$n^+ + n^- -2$ depending on whether $z$ is in one sector or the other of $f(U)\smm \beta$.

\ENDTHM

Notice that only the parity but not the size of $j$ comes into account, and $n^+-n^-=0,1$ or $-1$.

\beginp By Lemma \ref{Smooth} we can assume $z_0=0$ and $f$ takes the form
$f(z)=p(z)+\overline{q(z)}$ with
$$p(z)=z^m+ bz^{m+1} + O(z^{m+2}),\quad q(z)=z^m,\quad |b|=  1.$$

In this case $\psi(z_0)=1$. 
From lemma \ref{curves}, we know that $t\mapsto \beta(t)$ is locally injective and the local shape of $\beta$
corresponds to that of $  u(t^j+it^{j+1})$. Therefore $\beta$ is a convex curve on one half plane if $j$ is odd and
is a cusp of the first kind tangent to $\R$ if $j$ is even, then has its tangent lines on the right. See Figure \ref{beta}.

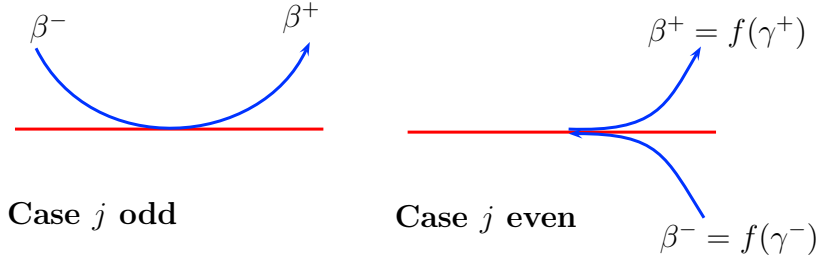
\begin{figure}\hspace{2cm}
\scalebox{1}
{
\begin{pspicture}(0,-1.6780468)(11.592149,1.6780468)
\definecolor{color215}{rgb}{0.0,0.2,1.0}
\psline[linewidth=0.04cm,linecolor=red](5.470254,-0.040351562)(9.570254,-0.040351562)
\psbezier[linewidth=0.04,linecolor=color215,arrowsize=0.05291667cm 2.0,arrowlength=1.4,arrowinset=0.4]{->}(7.610254,0.0)(8.770254,-0.040351562)(8.910254,0.29964843)(9.370254,1.0996485)
\usefont{T1}{ptm}{m}{n}
\rput(9.721709,1.2846484){$\beta^+=f(\gamma^+)$}
\usefont{T1}{ptm}{m}{n}
\rput(9.881709,-1.4553516){$\beta^-=f(\gamma^-)$}
\psbezier[linewidth=0.04,linecolor=color215,arrowsize=0.05291667cm 2.0,arrowlength=1.4,arrowinset=0.4]{<-}(7.610254,-0.06035156)(8.810254,-0.06035156)(8.890254,-0.34035155)(9.410254,-1.1803516)
\usefont{T1}{ptm}{m}{n}
\rput(6.5099807,-1.1953516){{\bf Case $j$ even}}
\psline[linewidth=0.04cm,linecolor=red](0.25025392,0.0)(4.350254,0.0)
\psbezier[linewidth=0.04,linecolor=color215,arrowsize=0.05291667cm 2.0,arrowlength=1.4,arrowinset=0.4]{->}(0.5302539,1.0754099)(1.2702539,-0.42035156)(3.510254,-0.25830027)(4.1702538,1.1596484)
\usefont{T1}{ptm}{m}{n}
\rput(4.061709,1.4846485){$\beta^+$}
\usefont{T1}{ptm}{m}{n}
\rput(0.701709,1.3446485){$\beta^-$}
\usefont{T1}{ptm}{m}{n}
\rput(1.2799804,-1.1553515){{\bf Case $j$ odd}}
\end{pspicture} 
}
\caption{The shape of the critical value curve}\label{beta}

\end{figure}

Write $f(z)=p(z)-q(z) +2 \Re q(z) = b(\kappa(z))^{m+1} + 2 \Re q(z)$ with $\kappa$ a holomorphic map tangent
to the identity at $0$. We may take $\kappa(z)$ as coordinate and transform $f$ into the following holomorphic+real normal form
 \REFEQN{Real-translation} f(z) =e^{i\theta}z^{m+1} + r(z)= F(z)+ r(z)\quad\text{with}\  F(z)=e^{i\theta}z^{m+1},  r(z)=2\Re(z^m+O(z^{m+1})).
\ENDEQN

Claim 0. In this form the critical value curve $\beta$ is either a convex curve on one half plane or
is a cusp of the first kind tangent to $\R$.

Proof. We have only changed the variable in the source plane. So this new normal form has the same critical value curve as before.  

Claim 1. We give here a specific proof to be compared to lemma \ref{isotopy}. For a small enough round circle  $C=\{|z|=s\}$ in the range,   its preimage by $f$ contains a Jordan curve connected component bounding a neighborhood $U$ of $0$, with $f(U)\subset D_s$ (not necessarily equal) and $f:U\to D_s$ proper (see Figure \ref{U}).

Notice that the tangent of $\g$ at $0$ depends on the choice of $\theta$ in $b=e^{i\theta}$, whereas the tangent of $\beta$
at $0$ does not depend on $\theta$.

\begin{figure}\hspace{2cm}
\scalebox{1} {
\begin{pspicture}(0,-2.2791991)(10.141016,2.299199)
\psline[linewidth=0.04cm,linecolor=red](0.26101562,-1.1991992)(3.7010157,0.9208008)
\psline[linewidth=0.04cm,linecolor=red](1.9010156,1.8408008)(1.9210156,-2.2591991)
\psline[linewidth=0.04cm,linecolor=red](0.12101562,0.84080076)(3.7410157,-1.1991992)
\psbezier[linewidth=0.04,arrowsize=0.05291667cm 2.0,arrowlength=1.4,arrowinset=0.4]{<-}(2.5210156,1.8208008)(2.5210156,1.0208008)(1.3810157,-1.5591992)(0.9810156,-1.7791992)
\psline[linewidth=0.04cm,linecolor=red](6.0210156,-0.2791992)(10.121016,-0.2791992)
\usefont{T1}{ptm}{m}{n}
\rput(2.7824707,2.1058009){$\gamma^+$}
\usefont{T1}{ptm}{m}{n}
\rput(0.8624707,-2.0541992){$\gamma^-$}
\psdots[dotsize=0.12](8.121016,-0.25919923)
\pscircle[linewidth=0.04,dimen=outer](8.141016,-0.23919922){1.3}
\psbezier[linewidth=0.04](1.4410156,1.3408008)(2.4029737,1.6139978)(3.8369179,0.49721578)(3.5810156,-0.47919923)(3.3251133,-1.4556142)(1.5886531,-1.8633064)(0.84101564,-1.1991992)(0.093378216,-0.53509206)(0.47905752,1.0676037)(1.4410156,1.3408008)
\usefont{T1}{ptm}{m}{n}
\rput(0.7924707,-0.05419922){$U$}
\usefont{T1}{ptm}{m}{n}
\rput(9.662471,-0.8341992){$C$}
\psbezier[linewidth=0.04,arrowsize=0.05291667cm 2.0,arrowlength=1.4,arrowinset=0.4]{->}(4.2210155,0.5608008)(4.2610154,0.6808008)(5.4410157,1.0008007)(6.0210156,0.42080078)
\usefont{T1}{ptm}{m}{n}
\rput(5.102471,1.0658008){$f$}
\end{pspicture} 
}
\caption{The domain $U$ and $F^{-1}(\R)$}\label{U}
\end{figure}
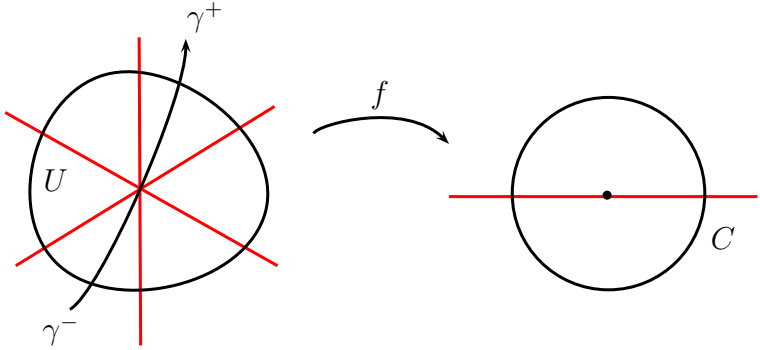

Proof.  
By assumption on $j<\infty$ the point $0$ is an isolated point in $f^{-1}(0)$.
So there is $r>0$ such that $\{|z|\le r\}$  is contained in the domain of definition $\Omega$ of $f$ and $0\notin f(\{|z|=r\})$. 

There is therefore a small round open disc $D$ centred at $0$ in the range such that $D\cap f(\{|z|=r\})=\emptyset$.

Let $W$ be an open connected subset of  $D$ containing $0$.

As $f$ is continuous $f^{-1}(W)$ is open in $\Omega$. Let $V$ be the connected component of $f^{-1}(W)$ containing $0$.
Then $V$ is an open neighborhood of $0$ with $V\subset \{|z|<r\}\subset\subset \Omega$.

We now claim that $f|_V:  V\to W$ is proper.

Proof. Let $V\ni z_n\to z\in \partial V$. We need to show $f(z_n)\to \partial W$. 
As $z\in \partial V\subset \Omega$ the map $f$ is defined and continuous at $z$. It follows that
$W\ni f(z_n)\to f(z)\in \overline W=W\sqcup \partial W$. If  $f(z)\in W$, then by continuity
$f$ maps a small disc neighborhood $B$ of $z$ into $W$, consequently $$B\cup V \text{ is }
\left\{\begin{array}{l} \text{connected}\\
        \text{strictly larger than $V$, and }\\
        \text{a subset of 
$ f^{-1}(W)$.}
       \end{array}\right.$$
 This contradicts the choice of  $V$ as a connected component of $f^{-1}(W)$ and
 ends the proof of the claim. 
 We now choose $W$ a small enough disc such that $t\to |\beta(t)|$ is strictly increasing  (resp. decreasing) as along as $t>0$ (resp. $t<0$) and $\beta(t)\in W$ and consider the proper map $f:=f|_V: V\to W$. 

Fix now $C=\{|z|=s\}$ contained in $W$ in the range.
The map $f$ is a local homeomorphism at every point of $f^{-1}C\smm \g$. Due to the local fold model 
at points of $\g^*$ we may conclude that $f^{-1}C$ is a 1-dimensional topological manifold, which is actually piecewise smooth. It is also compact by properness, so has only finitely many components, each is a Jordan curve. 

Let $I$ be an island, i.e. an open Jordan domain in $V$ bounded by a curve in $f^{-1}C$ . We claim that $f(I)\subset D_s:=\{|z|<s\}$.

Assume  $f(I)\smm \overline D_s\ne \emptyset$ . Then $|f|$ on the compact set $\overline I$ reaches its maximum at an interior point $x\in I$.
Then $x$ can not be outside $ \g$ as $f$ is locally open outside $\g$. But if $x\in \g$ then $f(x)\in \beta$ 
and $|f|$ restricted to $\g$ can not reach a local maximum since $|\beta(t)|$ is locally monotone. This is not possible.

So $f(I)\subset \overline D_s$. But if for some $x\in I$ we have $f(x)\in C=\partial D_s$, then $I$ contains a  component (so a Jordan curve) of $f^{-1}C$. Choose a point $x'$ in this curve but disjoint from $\g$. Then $f$ is a local homeomorphism
on a small disc $B$ centred at $x'$ with $B\subset I$ and $f(B)$ contains points outside $\overline D_s$. This is not possible by the previous paragraph. So we may conclude that $f(I)\subset D_s$. 

We claim now $0\in I$. Otherwise $0\not\in f(I)$ and we may argue as above using the minimum of $|f|$ on $\overline I$ to 
reach a contradiction.

It follows that $f^{-1}C$ has only one component in $V$ bounding a Jordan domain $U$ containing $0$ and $f(U)\subset D_s$. As $f$ maps the boundary into the boundary (not necessarily onto), $f: U\to D_s$ is proper.

Claim 2. The set $F^{-1}\R^*$ is a regular star of $2(m+1)$   radial branches from $0$ to $\infty$ and $F^{-1}\R\cap U$ is connected (see Figure \ref{U}).

Otherwise there is a segment $L\subset F^{-1}\R^*$ connecting two boundary points of $U$. As $f(s)=F(s)+r(s)$ with $r$ real, $f(L)\subset \R$. But $f^{-1}(0)=0$. So $f( L)$ is a segment in $\R^*$ by Intermediate Value Theorem. Now as $f$ has no turning points (critical points) in $ L$,  it maps $ L$ bijectively onto a 
real segment with constant sign, and the two ends are in $f(\partial U)=C$. This contradicts the choice that $C$ is a round circle.

Claim 3.  Each sector $S$ of $U\smm F^{-1}\R$ is mapped by both $f$ and $F$ into the same upper half plane.
Each branch $\ell$ of  $F^{-1}\R^*$ is mapped by $f$ to a real segment with constant sign (but not necessarily equal
to the sign of $F(\ell)$).
Two  consecutive branches on the same side of $\g$ have  images under $f$ with opposite signs, and
two consecutive branches separated by $\g$ have  images under $f$ with the same sign.

Proof. As $f(s)=F(s)+r(s)$ with $r$ real, and $F(S)$ is either on the upper or lower half plane, the same is true for $f(S)$ with the same imaginary sign.

 The fact that $F(\ell)\subset \R^*$ implies $f(\ell)\subset \R$.
But $f^{-1}(0)=0$. So $f(\ell)$ is a segment in $\R^*$ by Intermediate Value Theorem. Now as $f$ has no turning points (critical points) in $\ell$,  it maps $\ell$ bijectively onto a 
real segment with constant sign.

We now prove by contradiction that two consecutive branches on the same side of $\g$ have  images under $f$ with opposite signs.
Let $W$ be a small closed sector neighborhood of $0$ bounded by two consecutive branches on the same side of $\g$
and a small arc $\al$. Assume $f$ maps the two branches to the same segment in $\R$, say $[0,\ep]$.
As $W\cap f^{-1}\R^*=W\cap F^{-1}\R^*=\emptyset$, the connected set $f(W)$ is disjoint from $\R^-$. And $f(\al)$ is disjoint
from $0$.  Since $f(W)$ is not entirely contained in $\R^+$, one can find $v\in \partial f(W)\smm \Big(f(\al)\cup \R^+\cup \{0\}\Big)$.
So $v=f(w)$ for some interior point $w$ of $W$. This contradicts that $f$ is a local homeomorphism.

We may prove similarly that two  consecutive branches of $F^{-1}\R^*$ separated by $\g$ have  images under $f$ with the same sign, using the fact that $f$ realises a fold along $\g^*$.

Claim 4. Let  $S$ be a sector of $U\smm F^{-1}\R$ disjoint from $\g$. Then $f$ maps $S$ homeomorphically onto
one of the half discs $\{|z|<s, \Im z >0\}, \{|z|<s, \Im z <0\}$, and  in $S$   the  number of branches of $f^{-1}(f(\g))$ is equal 
to the number of branches of $F^{-1}(F(\g))$ (see Figure \ref{co-critical}).

Proof. The previous claim says that $f$ is a local homeomorphism on $S$, and $f(S)$ is contained
in one of the half discs, say $\{|z|<s, \Im z >0\}$. We also know that $f:S\to \{|z|<s, \Im z >0\}$ is proper, so is in fact a covering. As $S$ is simply connected, we conclude that $f$ on $S$ is a homeomorphism onto
its image. We also need to prove that $f(S)$ is one of the half discs bounded by $C\cup \R$.  

For $t\in ]0,\ep[$, set $\g^\pm(t)=\g(\pm t)$. Consider 
 $\de^\pm(t)=F(\g^\pm( t))$ and $\beta^\pm(t)=f(\g^\pm(t))$, 

By \Ref{Real-translation} we know that $\de^-(t)$ and $\beta^-(t)$ are in the same half plane of $\C\smm \R$,
idem for the pair $\de^+(t)$ and $\beta^+(t)$. Comparing with the shape of $\beta$ relative to $\R$ 
we know that
$\de^\pm(t)$ are in the same half plane if $\beta$ is convex and in opposite half planes otherwise.

Claim 5. The map $f$ sends each  $S$ of the two sectors  of $U\smm F^{-1}\R$ intersecting  $\g$  onto
one small sector  $\chi$ with $0$ angle at $0$  of  $\C\smm (C\cup\beta\cup \R)$, and $S\cap f^{-1}(\beta)\subset \g$ (see Figure \ref{folding-side}).

\begin{figure}\hspace{2cm}
\scalebox{1} 
{
\begin{pspicture}(0,-4.859199)(11.041895,4.879199)
\definecolor{color2287}{rgb}{0.0,0.2,1.0}
\definecolor{color2854b}{rgb}{0.8,0.8,0.8}
\psline[linewidth=0.04cm,linecolor=red,linestyle=dotted,dotsep=0.16cm](0.0,2.460801)(4.92,2.4408007)
\psline[linewidth=0.04cm,linecolor=red,linestyle=dotted,dotsep=0.16cm](2.36,4.4408007)(2.38,0.3408008)
\psline[linewidth=0.04cm,linecolor=red,linestyle=dotted,dotsep=0.16cm](0.66,0.9208008)(4.18,4.0608006)
\psline[linewidth=0.04cm,linecolor=red,linestyle=dotted,dotsep=0.16cm](0.82,4.0408006)(3.98,0.8008008)
\psbezier[linewidth=0.04,arrowsize=0.05291667cm 2.0,arrowlength=1.4,arrowinset=0.4]{<-}(2.96,4.4408007)(2.96,3.6408007)(1.82,1.0608008)(1.42,0.84080076)
\psbezier[linewidth=0.04,linestyle=dashed,dash=0.16cm 0.16cm](4.45,2.0108008)(3.18,1.8608007)(1.62,2.6808007)(0.75,3.5108008)
\psbezier[linewidth=0.04,linestyle=dashed,dash=0.16cm 0.16cm](0.55,2.9308007)(1.35,2.9308007)(3.94,1.8008008)(4.15,1.3908008)
\psline[linewidth=0.04cm,linecolor=red,linestyle=dotted,dotsep=0.16cm](5.56,2.4408007)(9.66,2.4408007)
\psbezier[linewidth=0.04,linecolor=color2287,arrowsize=0.05291667cm 2.0,arrowlength=1.4,arrowinset=0.4]{->}(5.82,3.4976285)(6.56,2.0208008)(8.8,2.1808007)(9.46,3.5808008)
\usefont{T1}{ptm}{m}{n}
\rput(9.831455,3.7458007){$\beta^+$}
\usefont{T1}{ptm}{m}{n}
\rput(5.8114552,3.7058008){$\beta^-$}
\usefont{T1}{ptm}{m}{n}
\rput(3.201455,4.6858006){$\gamma^+$}
\usefont{T1}{ptm}{m}{n}
\rput(1.4814551,0.64580077){$\gamma^-$}
\usefont{T1}{ptm}{m}{n}
\rput(7.5497265,1.4658008){{\bf Case $m$ odd and $j$ odd}}
\psline[linewidth=0.04cm,linecolor=red,linestyle=dotted,dotsep=0.16cm](1.32,-3.7791991)(4.84,-1.5991992)
\psline[linewidth=0.04cm,linecolor=red,linestyle=dotted,dotsep=0.16cm](2.96,-0.7391992)(2.98,-4.839199)
\psline[linewidth=0.04cm,linecolor=red,linestyle=dotted,dotsep=0.16cm](1.16,-1.6991992)(4.8,-3.7791991)
\psbezier[linewidth=0.04,arrowsize=0.05291667cm 2.0,arrowlength=1.4,arrowinset=0.4]{<-}(4.62,-0.9591992)(4.42,-1.8391992)(2.48,-3.3191993)(1.18,-3.4591992)
\psbezier[linewidth=0.04,linestyle=dashed,dash=0.16cm 0.16cm](3.46,-4.699199)(3.1,-3.7791991)(3.2,-4.219199)(3.0,-2.799199)
\psbezier[linewidth=0.04,linestyle=dashed,dash=0.16cm 0.16cm](2.96,-2.7591991)(3.8380961,-3.381902)(3.913438,-3.3451452)(4.4,-4.279199)
\psline[linewidth=0.04cm,linecolor=red,linestyle=dotted,dotsep=0.16cm](6.02,-2.9591992)(10.12,-2.9591992)
\psbezier[linewidth=0.04,linecolor=color2287,arrowsize=0.05291667cm 2.0,arrowlength=1.4,arrowinset=0.4]{->}(6.0,-2.039199)(7.38,-3.3191993)(8.96,-3.2591991)(10.02,-1.8791993)
\usefont{T1}{ptm}{m}{n}
\rput(10.291455,-1.6341993){$\beta^+$}
\usefont{T1}{ptm}{m}{n}
\rput(5.911455,-1.8141992){$\beta^-$}
\usefont{T1}{ptm}{m}{n}
\rput(4.741455,-0.6741992){$\gamma^+$}
\usefont{T1}{ptm}{m}{n}
\rput(1.2214551,-3.0741992){$\gamma^-$}
\usefont{T1}{ptm}{m}{n}
\rput(8.039726,-4.0141993){{\bf Case $m$ even and $j$ odd}}
\psbezier[linewidth=0.04](2.36,4.180801)(3.36,4.200801)(3.34,4.320801)(3.82,3.8208008)
\psline[linewidth=0.04cm](2.64,4.140801)(3.36,3.4408007)
\psline[linewidth=0.04cm](2.4,3.9408007)(3.12,3.2408009)
\psline[linewidth=0.04cm](2.42,3.4408007)(2.88,3.0008008)
\psline[linewidth=0.04cm](2.42,3.0608008)(2.66,2.8408008)
\psline[linewidth=0.04cm](3.14,4.160801)(3.64,3.6608007)
\psbezier[linewidth=0.04](9.28,3.3008008)(9.44,3.1408007)(9.66,2.7408009)(9.46,2.460801)
\psline[linewidth=0.04cm](9.54,2.8008008)(9.24,2.4808009)
\psline[linewidth=0.04cm](9.5,3.0008008)(8.92,2.4808009)
\psline[linewidth=0.04cm](9.38,3.1808007)(8.58,2.460801)
\psline[linewidth=0.04cm](8.56,2.6808007)(8.32,2.460801)
\pscircle[linewidth=0.04,dimen=outer,fillstyle=solid,fillcolor=color2854b](1.72,1.2208008){0.2}
\psbezier[linewidth=0.04,fillstyle=solid,fillcolor=color2854b](6.02,3.1408007)(5.96,2.8008008)(6.08,2.700801)(6.26,2.9008007)
\usefont{T1}{ptm}{m}{n}
\rput(9.991455,3.0058007){$\chi$}
\end{pspicture} 
}
\caption{The folding sides for harmonic maps $f$}\label{folding-side}
\end{figure}
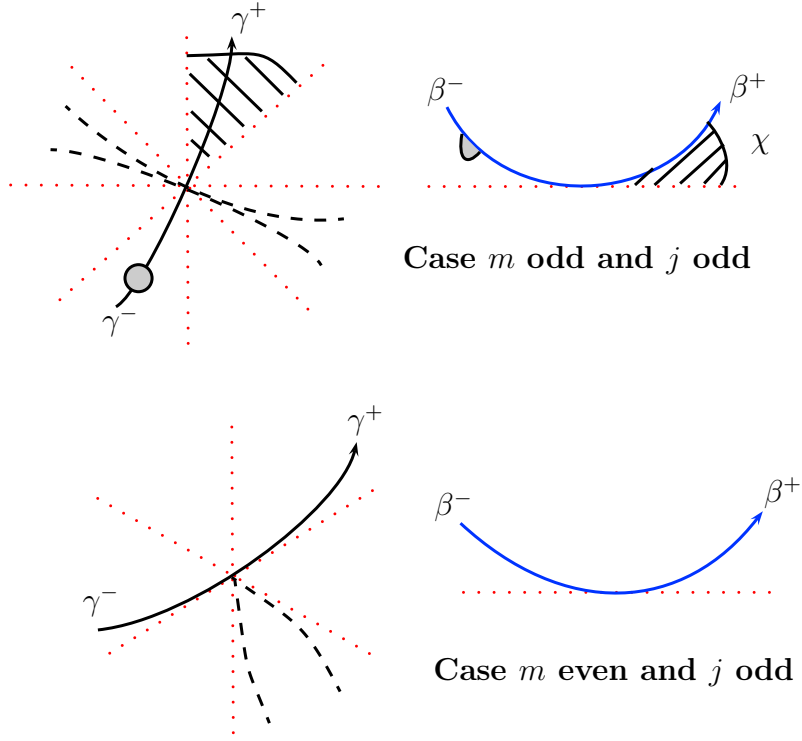

This is due to the harmonicness: $f$ folds a small neighborhood of $z\in \g^*$ onto a half neighborhood of
$f(z)$ on the concave side of $\beta$ (see Figure \ref{folding-side}). As $\chi$ does not contain the other branch of $\beta$, the preimage
$S$ contains no other co-critical points than $\g$. 

Claim 6. The critical curve $\g$ separates the  branches of $F^{-1}\R^*$ into two parts whose numbers depend on the shape of $\beta$, by the following table:

$$\begin{array}{|c||c|c|} \hline { \lv\#\text{right \mystrut branches of $F^{-1}\R^*$} \\ \#\text{left branches of $F^{-1}\R^*$}\rv}   & \text{$\beta$ convex}  &\text{$\beta$ cusp} \\ \hline\hline
m\text{ odd}
    & \text{equal number}  &   \text{differ by 2}  \\   \hline
m\text{ even} & 
  \text{differ by 2}   &\text{equal number} \\ \hline
\end{array} $$

Proof. For $t\in ]0,\ep[$, we have $\g^\pm(t)=\g(\pm t)$,
 $\de^\pm(t)=F(\g^\pm( t))$ and $\beta^\pm(t)=f(\g^\pm(t))$.
We need to know the relative positions between  $\de^\pm(t)$ and $\R$ in order to get the relative  positions
between $\g\subset F^{-1}(\de^\pm(t))$ and $F^{-1}\R^*$.

We know that $\de^\pm(t)$ are in the same half plane if $\beta$ is convex and in opposite half planes otherwise.

On the other hand, the two curves  $\g^\pm(t), t\in [0,\ep[$ make an angle $\pi$ at $\g(0)$.
As $F(z)= e^{i\theta}z^{m+1}$,  $$\text{angle}_0(
 \de^\pm(t))= (m+1)\cdot \text{angle}_0(\g^\pm( t))=(m+1)\pi\!\!\!\mod\!2\pi =\left\{\begin{array}{ll} 0 & \text{if $m$ is odd}\\
 \pi & \text{if $m$ is even.}\end{array}\right.$$
  
Now pullback these shapes by $F(z)=e^{i\theta} z^{m+1}$, we get the claim. See Figure \ref{cases}.

Claim 7. In any case, the number of sectors in $U\smm f^{-1}\beta$ is odd  in each side of $\g$. Denoting them by $2n^\pm-1$, with $+$ for the right-side of $\g$ and $-$ the left side, one can related them to the numbers of branches of $F^{-1}\R^*$ separated by $\g$ by: 
$$\begin{array}{|r|c|c|} \hline &\mystrut \text{$\beta$ convex, $j$ odd} &\text{$\beta$ cusp, $j$ even}\\ \hline
\lv \mystrut 2n^+-1\\ 2n^--1\rv = & \lv\#\text{right branches of $F^{-1}\R^*$}-1 \mystrut \\ \#\text{left branches of $F^{-1}\R^*$}-1\rv &  \lv\#\text{right \mystrut branches of $F^{-1}\R^*$} \\ \#\text{left branches of $F^{-1}\R^*$}\rv \\
\hline
\end{array} $$

Proof. The shape of $\beta$ is determined by the parity of $j$ in its order-pair  $(j,j+1)$: If $j$ is odd then $\beta$ is convex, 
if $j$ is even then $\beta$ is a cusp. In the following only the shape of $\beta$ is relevant, but not the value of $j$.
It follows that if $m+j$ is odd, $F^{-1}\R$ contains the tangent line of $\g$ at $0$.

 See Figure \ref{co-critical}.

Now as the total number of branches of $F^{-1}\R^* $ is $2(m+1)$, we get, by Claim 6, 
$$\begin{array}{|c|c|c|} \hline { \lv\#\text{right \mystrut branches of $F^{-1}\R^*$} \\ \#\text{left branches of $F^{-1}\R^*$}\rv}   & \text{$\beta$ convex, $j$ odd}  &\text{$\beta$ cusp, $j$ even} \\ \hline 
m\text{ odd}
    & {\lv m+1\mystrut \\ m+1\rv}  & 
\lv m+2 \\ m \rv \quad \text{or}\quad {\lv m  \\ m+2 \rv}  \\   \hline
m\text{ even} & 
\lv m+2\mystrut \\ m \rv \quad \text{or}\quad {\lv m  \\ m+2 \rv}   & {\lv m+1 \\ m+1\rv} \\ \hline
\end{array} $$

We get \Ref{n-plus}. 

Claim 8. Now we forget about $F^{-1}\R$ and consider only the sectors in $U$ partitioned by $f^{-1}\beta$. The same arguments
as above show that $f$ maps each sector homeomorphically onto one of the two sectors in $D_s\smm \beta$ in the range. 

 To construct coordinate changes $h_1, h_2$ from $f$ to $R_{n^+, n^-}$, one proceeds as follows:

Define at first an orientation preserving  homeomorphisms  $h_2: \overline D_s\to \overline \D$ mapping $0$ to $0$ and $\beta\cap  \overline D_s$
onto $[-1,1]$.  Note that $R_{n^+, n^-}^{-1}[-1,1]$ partitions $\overline \D$ into the same number of sectors as the partition of $U$ by $f^{-1}\beta$. We just need now to construct $h_1$ sector on sector so that $R_{n^+, n^-}\circ h_1=h_2\circ f$ on that sector and
$h_1$ is an orientation preserving  
mapping from $\g\cap \overline D_s$ onto $[-1,1]$. We can see that  $h_1$ is a homeomorphism from $U$ to $\D$.
\qed

Notice that the local topological degree of $f$ can be expressed in the following table:
$$\begin{array}{|c||c|c|c|c|}\hline &\multicolumn{2}{|c|}{\beta(t)  \text{ convex, $m\le j$ odd} }&\multicolumn{2}{|c|}{ \beta(t) \text{ cusp, $m\le j$ even} }\\ \hline
&m \text{ odd}&m \text{ even} &m \text{ odd}  &m \text{ even} \\ \hline
 f_{z=0}\sim \left(\!\!\begin{array}{c} \mystrut z^{2n^+ -1}\\\barz^{2n^- -1}  
\end{array}\!\!\right) &  \left(\!\!\begin{array}{c}z^{ m}\\ \overline z^{m} \end{array}\!\!\right) &  \left(\!\!\begin{array}{c}z^{ {m+1} }\\  \barz^{m-1} \end{array}\!\!\right) 
 \text{ or }  \left(\!\!\begin{array}{c}z^{ m-1}\\  \barz^{{m+1}  } \end{array}\!\!\right)  \mystrut  & 
  \left(\!\!\begin{array}{c}z^{ m+2}\\ \overline z^{m} \end{array}\!\!\right) \text{ or } 
   \left(\!\!\begin{array}{c}z^{ m}\\ \overline z^{m+2} \end{array}\!\!\right) &   \left(\!\!\begin{array}{c}z^{ m+1}\\  \barz^{m+1 } \end{array}\!\!\right)\\ \hline 
   \mystrut \deg(f,0)= & 0 & \pm 1 & \pm 1 & 0\\ \hline
   \# f^{-1}(z) = & m+1, m-1 & m+1\mystrut, m-1 & m+2,m& m+2,m \\ \hline
   \mu(f,0)= &\multicolumn{4}{c|}{j+m^2}\\ \hline
\end{array}$$

\REFCOR{generic again} In the generic case $f(z)=(z+bz^2+O(z^3))^m-\zbar^m$ with $(-b^2)^m\ne 1$, we have
$$\begin{array}{|c||c|c|}\hline
&\beta(t)  \text{ convex, $m=j$ odd} &\beta(t) \text{ cusp, $m=j$ even}\\ \hline
f_{z=0}\sim  \left(\!\!\begin{array}{c} \mystrut z^{2n^+ -1}\\\barz^{2n^- -1}  
\end{array}\!\!\right) &  \left(\!\!\begin{array}{c}z^{ m}\\ \overline z^{m} \end{array}\!\!\right) & \left(\!\!\begin{array}{c}z^{ m+1}\\  \barz^{m+1 } \end{array}\!\!\right)\\ \hline 
   \mystrut \deg (f,0)= & 0 &  0\\ \hline
   \# f^{-1}(z) = & m+1, m-1 & m+2,m \\ \hline
   \mu(f,0)= &\multicolumn{2}{c|}{m+m^2} \\ \hline
 \end{array}$$
\ENDCOR

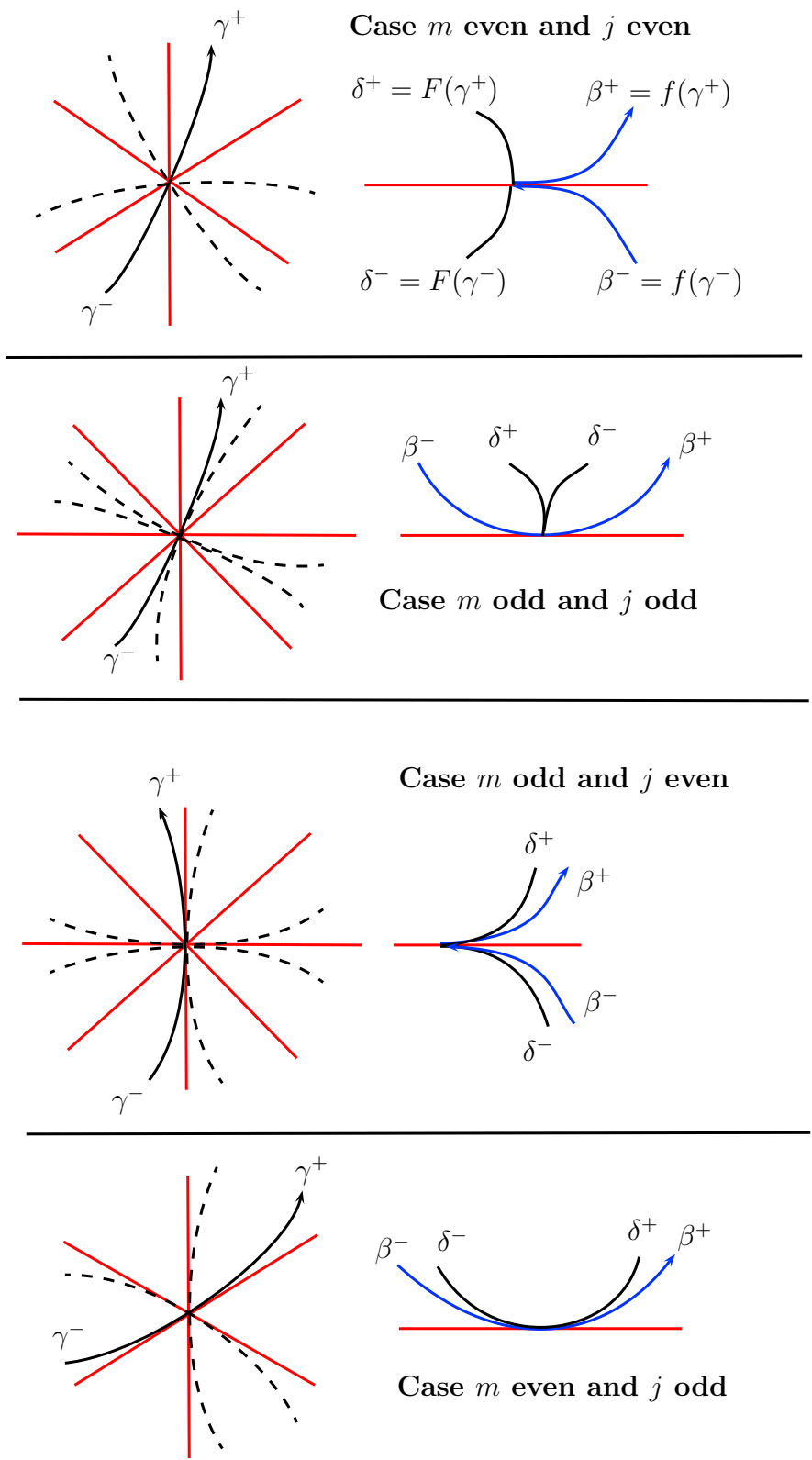
\begin{figure}\hspace{2cm}
\scalebox{1} 
{
\begin{pspicture}(0,-10.49)(11.68,10.51)
\definecolor{color1623}{rgb}{0.0,0.2,1.0}
\psline[linewidth=0.04cm,linecolor=red](0.72,7.01)(4.28,9.23)
\psline[linewidth=0.04cm,linecolor=red](2.36,10.05)(2.38,5.95)
\psline[linewidth=0.04cm,linecolor=red](0.7,9.21)(4.1,6.85)
\psbezier[linewidth=0.04,arrowsize=0.05291667cm 2.0,arrowlength=1.4,arrowinset=0.4]{<-}(2.98,10.03)(2.98,9.23)(1.84,6.65)(1.44,6.43)
\psbezier[linewidth=0.04,linestyle=dashed,dash=0.16cm 0.16cm](3.7,6.47)(3.18,6.69)(1.66,8.87)(1.5,9.87)
\psbezier[linewidth=0.04,linestyle=dashed,dash=0.16cm 0.16cm](0.44,7.53)(1.68,8.17)(3.9,8.09)(4.48,7.87)
\psline[linewidth=0.04cm,linecolor=red](5.2,7.99)(9.3,7.99)
\psbezier[linewidth=0.04,linecolor=color1623,arrowsize=0.05291667cm 2.0,arrowlength=1.4,arrowinset=0.4]{->}(7.34,8.03)(8.5,7.99)(8.64,8.33)(9.1,9.13)
\psbezier[linewidth=0.04](7.355484,7.983008)(7.34,8.883008)(6.984516,8.95)(6.82,9.05)
\psbezier[linewidth=0.04](7.32,7.97)(7.26,7.29)(7.02,7.23)(6.68,6.93)
\usefont{T1}{ptm}{m}{n}
\rput(9.47,9.315){$\beta^+=f(\gamma^+)$}
\usefont{T1}{ptm}{m}{n}
\rput(9.63,6.575){$\beta^-=f(\gamma^-)$}
\usefont{T1}{ptm}{m}{n}
\rput(6.08,9.355){$\delta^+=F(\gamma^+)$}
\usefont{T1}{ptm}{m}{n}
\rput(6.2,6.575){$\delta^-=F(\gamma^-)$}
\usefont{T1}{ptm}{m}{n}
\rput(3.26,10.315){$\gamma^+$}
\usefont{T1}{ptm}{m}{n}
\rput(1.34,6.155){$\gamma^-$}
\psbezier[linewidth=0.04,linecolor=color1623,arrowsize=0.05291667cm 2.0,arrowlength=1.4,arrowinset=0.4]{<-}(7.34,7.97)(8.54,7.97)(8.62,7.69)(9.14,6.85)
\usefont{T1}{ptm}{m}{n}
\rput(7.46,10.275){{\bf Case $m$ even and $j$ even}}
\psline[linewidth=0.04cm,linecolor=red](0.16,2.93)(5.08,2.91)
\psline[linewidth=0.04cm,linecolor=red](2.52,4.91)(2.54,0.81)
\psline[linewidth=0.04cm,linecolor=red](0.82,1.39)(4.34,4.53)
\psline[linewidth=0.04cm,linecolor=red](0.98,4.51)(4.14,1.27)
\psbezier[linewidth=0.04,arrowsize=0.05291667cm 2.0,arrowlength=1.4,arrowinset=0.4]{<-}(3.12,4.91)(3.12,4.11)(1.98,1.53)(1.58,1.31)
\psbezier[linewidth=0.04,linestyle=dashed,dash=0.16cm 0.16cm](2.2,1.09)(2.08,2.35)(2.9,3.85)(3.7,4.79)
\psbezier[linewidth=0.04,linestyle=dashed,dash=0.16cm 0.16cm](4.61,2.48)(3.34,2.33)(1.78,3.15)(0.91,3.98)
\psbezier[linewidth=0.04,linestyle=dashed,dash=0.16cm 0.16cm](0.71,3.4)(1.51,3.4)(4.1,2.27)(4.31,1.86)
\psline[linewidth=0.04cm,linecolor=red](5.72,2.91)(9.82,2.91)
\psbezier[linewidth=0.04,linecolor=color1623,arrowsize=0.05291667cm 2.0,arrowlength=1.4,arrowinset=0.4]{->}(5.98,3.9668276)(6.72,2.49)(8.96,2.65)(9.62,4.05)
\psbezier[linewidth=0.04](7.78,2.9030082)(7.92,3.65)(7.44,3.83)(7.3,3.95)
\psbezier[linewidth=0.04](7.8,2.97)(7.9,3.7238462)(8.1,3.686154)(8.44,3.95)
\usefont{T1}{ptm}{m}{n}
\rput(10.01,4.215){$\beta^+$}
\usefont{T1}{ptm}{m}{n}
\rput(5.99,4.175){$\beta^-$}
\usefont{T1}{ptm}{m}{n}
\rput(7.22,4.355){$\delta^+$}
\usefont{T1}{ptm}{m}{n}
\rput(8.66,4.375){$\delta^-$}
\usefont{T1}{ptm}{m}{n}
\rput(3.38,5.155){$\gamma^+$}
\usefont{T1}{ptm}{m}{n}
\rput(1.66,1.115){$\gamma^-$}
\usefont{T1}{ptm}{m}{n}
\rput(7.72,1.935){{\bf Case $m$ odd and $j$ odd}}
\psline[linewidth=0.04cm](0.0,5.49)(11.54,5.51)
\psline[linewidth=0.04cm](0.2,0.53)(11.64,0.51)
\psline[linewidth=0.04cm,linecolor=red](0.24,-3.01)(5.16,-3.03)
\psline[linewidth=0.04cm,linecolor=red](2.6,-1.03)(2.62,-5.13)
\psline[linewidth=0.04cm,linecolor=red](0.9,-4.55)(4.42,-1.41)
\psline[linewidth=0.04cm,linecolor=red](1.06,-1.43)(4.22,-4.67)
\psbezier[linewidth=0.04,arrowsize=0.05291667cm 2.0,arrowlength=1.4,arrowinset=0.4]{<-}(2.22,-1.03)(2.64,-2.01)(2.84,-3.99)(2.08,-4.99)
\psline[linewidth=0.04cm,linecolor=red](5.62,-3.03)(8.34,-3.03)
\psbezier[linewidth=0.04,linecolor=color1623,arrowsize=0.05291667cm 2.0,arrowlength=1.4,arrowinset=0.4]{->}(6.3,-3.01)(7.84,-2.93)(7.8,-2.37)(8.16,-1.89)
\psbezier[linewidth=0.04](6.3,-3.0369918)(7.44,-2.95)(7.56,-2.33)(7.68,-1.91)
\psbezier[linewidth=0.04](6.3,-3.05)(7.18,-3.03)(7.64,-3.49)(7.86,-4.21)
\usefont{T1}{ptm}{m}{n}
\rput(8.53,-2.125){$\beta^+$}
\usefont{T1}{ptm}{m}{n}
\rput(8.63,-3.865){$\beta^-$}
\usefont{T1}{ptm}{m}{n}
\rput(7.74,-1.545){$\delta^+$}
\usefont{T1}{ptm}{m}{n}
\rput(7.72,-4.545){$\delta^-$}
\usefont{T1}{ptm}{m}{n}
\rput(2.32,-0.665){$\gamma^+$}
\usefont{T1}{ptm}{m}{n}
\rput(1.78,-5.285){$\gamma^-$}
\psbezier[linewidth=0.04,linecolor=color1623,arrowsize=0.05291667cm 2.0,arrowlength=1.4,arrowinset=0.4]{<-}(6.38,-3.05)(7.92,-3.13)(7.88,-3.69)(8.24,-4.17)
\psbezier[linewidth=0.04,linestyle=dashed,dash=0.16cm 0.16cm](0.64,-3.43)(1.62,-3.01)(3.6,-2.81)(4.6,-3.57)
\psbezier[linewidth=0.04,linestyle=dashed,dash=0.16cm 0.16cm](0.64,-2.65)(1.62,-3.07)(3.6,-3.27)(4.6,-2.51)
\psbezier[linewidth=0.04,linestyle=dashed,dash=0.16cm 0.16cm](3.0,-1.07)(2.58,-2.05)(2.38,-4.03)(3.14,-5.03)
\usefont{T1}{ptm}{m}{n}
\rput(8.09,-0.645){{\bf Case $m$ odd and $j$ even}}
\psline[linewidth=0.04cm](0.3,-5.77)(11.66,-5.75)
\psline[linewidth=0.04cm,linecolor=red](1.0,-9.41)(4.52,-7.23)
\psline[linewidth=0.04cm,linecolor=red](2.64,-6.37)(2.66,-10.47)
\psline[linewidth=0.04cm,linecolor=red](0.84,-7.33)(4.48,-9.41)
\psbezier[linewidth=0.04,arrowsize=0.05291667cm 2.0,arrowlength=1.4,arrowinset=0.4]{<-}(4.3,-6.59)(4.1,-7.47)(2.16,-8.95)(0.86,-9.09)
\psbezier[linewidth=0.04,linestyle=dashed,dash=0.16cm 0.16cm](3.14,-10.33)(2.52,-9.41)(2.52,-7.45)(3.06,-6.25)
\psbezier[linewidth=0.04,linestyle=dashed,dash=0.16cm 0.16cm](0.86,-7.81)(2.2,-7.79)(3.78,-8.97)(4.22,-9.95)
\psline[linewidth=0.04cm,linecolor=red](5.7,-8.59)(9.8,-8.59)
\psbezier[linewidth=0.04,linecolor=color1623,arrowsize=0.05291667cm 2.0,arrowlength=1.4,arrowinset=0.4]{->}(5.68,-7.67)(7.06,-8.95)(8.64,-8.89)(9.7,-7.51)
\usefont{T1}{ptm}{m}{n}
\rput(9.99,-7.265){$\beta^+$}
\usefont{T1}{ptm}{m}{n}
\rput(5.61,-7.445){$\beta^-$}
\usefont{T1}{ptm}{m}{n}
\rput(9.24,-7.145){$\delta^+$}
\usefont{T1}{ptm}{m}{n}
\rput(6.48,-7.305){$\delta^-$}
\usefont{T1}{ptm}{m}{n}
\rput(4.44,-6.305){$\gamma^+$}
\usefont{T1}{ptm}{m}{n}
\rput(0.92,-8.705){$\gamma^-$}
\psbezier[linewidth=0.04](6.26,-7.69)(7.0,-9.01)(8.9,-8.753513)(9.18,-7.55)
\usefont{T1}{ptm}{m}{n}
\rput(8.07,-9.465){{\bf Case $m$ even and $j$ odd}}
\end{pspicture} 
}
\caption{The left hand figures are  $F^{-1}(\R)$ (in red) and $F^{-1}(F(\g))$ (in black).  The shape of $\beta^\pm$ is determined by the parity of $j$. The curves $\de^\pm$ are in the same half planes as
$\beta^\pm$ due to the fact that $F-f$ is real. The angle between $\de^\pm$ is determined by the parity of $m$, as $F(z)=e^{i\theta}z^{m+1}$. }\label{cases}\end{figure}

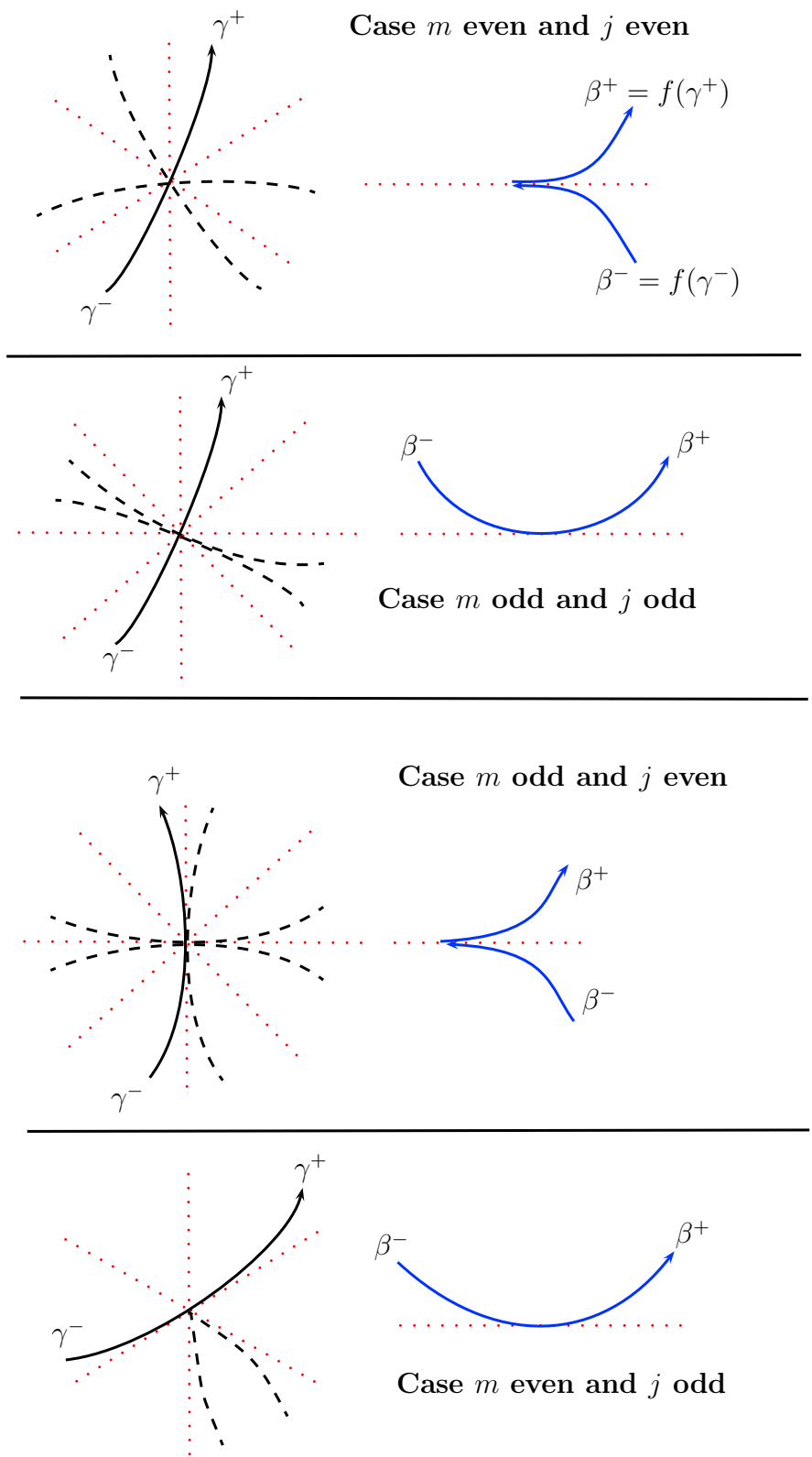
\begin{figure}\hspace{2cm}
\scalebox{1} 
{
\begin{pspicture}(0,-10.4892)(11.68,10.509199)
\definecolor{color4115}{rgb}{0.0,0.2,1.0}
\psline[linewidth=0.04cm,linecolor=red,linestyle=dotted,dotsep=0.16cm](0.72,7.010801)(4.28,9.230801)
\psline[linewidth=0.04cm,linecolor=red,linestyle=dotted,dotsep=0.16cm](2.36,10.0508)(2.38,5.950801)
\psline[linewidth=0.04cm,linecolor=red,linestyle=dotted,dotsep=0.16cm](0.7,9.210801)(4.1,6.850801)
\psbezier[linewidth=0.04,arrowsize=0.05291667cm 2.0,arrowlength=1.4,arrowinset=0.4]{<-}(2.98,10.030801)(2.98,9.230801)(1.84,6.6508007)(1.44,6.430801)
\psbezier[linewidth=0.04,linestyle=dashed,dash=0.16cm 0.16cm](3.7,6.470801)(3.18,6.6908007)(1.66,8.870801)(1.5,9.870801)
\psbezier[linewidth=0.04,linestyle=dashed,dash=0.16cm 0.16cm](0.44,7.530801)(1.68,8.170801)(3.9,8.090801)(4.48,7.870801)
\psline[linewidth=0.04cm,linecolor=red,linestyle=dotted,dotsep=0.16cm](5.2,7.990801)(9.3,7.990801)
\psbezier[linewidth=0.04,linecolor=color4115,arrowsize=0.05291667cm 2.0,arrowlength=1.4,arrowinset=0.4]{->}(7.34,8.030801)(8.5,7.990801)(8.64,8.330801)(9.1,9.130801)
\usefont{T1}{ptm}{m}{n}
\rput(9.451455,9.315801){$\beta^+=f(\gamma^+)$}
\usefont{T1}{ptm}{m}{n}
\rput(9.611455,6.575801){$\beta^-=f(\gamma^-)$}
\usefont{T1}{ptm}{m}{n}
\rput(3.241455,10.315801){$\gamma^+$}
\usefont{T1}{ptm}{m}{n}
\rput(1.3214551,6.155801){$\gamma^-$}
\psbezier[linewidth=0.04,linecolor=color4115,arrowsize=0.05291667cm 2.0,arrowlength=1.4,arrowinset=0.4]{<-}(7.34,7.970801)(8.54,7.970801)(8.62,7.6908007)(9.14,6.850801)
\usefont{T1}{ptm}{m}{n}
\rput(7.4497266,10.275801){{\bf Case $m$ even and $j$ even}}
\psline[linewidth=0.04cm,linecolor=red,linestyle=dotted,dotsep=0.16cm](0.16,2.9308007)(5.08,2.9108007)
\psline[linewidth=0.04cm,linecolor=red,linestyle=dotted,dotsep=0.16cm](2.52,4.910801)(2.54,0.8108008)
\psline[linewidth=0.04cm,linecolor=red,linestyle=dotted,dotsep=0.16cm](0.82,1.3908008)(4.34,4.530801)
\psline[linewidth=0.04cm,linecolor=red,linestyle=dotted,dotsep=0.16cm](0.98,4.510801)(4.14,1.2708008)
\psbezier[linewidth=0.04,arrowsize=0.05291667cm 2.0,arrowlength=1.4,arrowinset=0.4]{<-}(3.12,4.910801)(3.12,4.1108007)(1.98,1.5308008)(1.58,1.3108008)
\psbezier[linewidth=0.04,linestyle=dashed,dash=0.16cm 0.16cm](4.61,2.4808009)(3.34,2.3308008)(1.78,3.1508007)(0.91,3.9808009)
\psbezier[linewidth=0.04,linestyle=dashed,dash=0.16cm 0.16cm](0.71,3.4008007)(1.51,3.4008007)(4.1,2.2708008)(4.31,1.8608007)
\psline[linewidth=0.04cm,linecolor=red,linestyle=dotted,dotsep=0.16cm](5.72,2.9108007)(9.82,2.9108007)
\psbezier[linewidth=0.04,linecolor=color4115,arrowsize=0.05291667cm 2.0,arrowlength=1.4,arrowinset=0.4]{->}(5.98,3.9676285)(6.72,2.4908009)(8.96,2.6508007)(9.62,4.050801)
\usefont{T1}{ptm}{m}{n}
\rput(9.991455,4.215801){$\beta^+$}
\usefont{T1}{ptm}{m}{n}
\rput(5.971455,4.175801){$\beta^-$}
\usefont{T1}{ptm}{m}{n}
\rput(3.361455,5.155801){$\gamma^+$}
\usefont{T1}{ptm}{m}{n}
\rput(1.641455,1.1158007){$\gamma^-$}
\usefont{T1}{ptm}{m}{n}
\rput(7.7097263,1.9358008){{\bf Case $m$ odd and $j$ odd}}
\psline[linewidth=0.04cm](0.0,5.490801)(11.54,5.510801)
\psline[linewidth=0.04cm](0.2,0.53080076)(11.64,0.5108008)
\psline[linewidth=0.04cm,linecolor=red,linestyle=dotted,dotsep=0.16cm](0.24,-3.0091991)(5.16,-3.0291991)
\psline[linewidth=0.04cm,linecolor=red,linestyle=dotted,dotsep=0.16cm](2.6,-1.0291992)(2.62,-5.129199)
\psline[linewidth=0.04cm,linecolor=red,linestyle=dotted,dotsep=0.16cm](0.9,-4.549199)(4.42,-1.4091992)
\psline[linewidth=0.04cm,linecolor=red,linestyle=dotted,dotsep=0.16cm](1.06,-1.4291992)(4.22,-4.669199)
\psbezier[linewidth=0.04,arrowsize=0.05291667cm 2.0,arrowlength=1.4,arrowinset=0.4]{<-}(2.22,-1.0291992)(2.64,-2.0091991)(2.84,-3.9891992)(2.08,-4.989199)
\psline[linewidth=0.04cm,linecolor=red,linestyle=dotted,dotsep=0.16cm](5.62,-3.0291991)(8.34,-3.0291991)
\psbezier[linewidth=0.04,linecolor=color4115,arrowsize=0.05291667cm 2.0,arrowlength=1.4,arrowinset=0.4]{->}(6.3,-3.0091991)(7.84,-2.9291992)(7.8,-2.3691993)(8.16,-1.8891993)
\usefont{T1}{ptm}{m}{n}
\rput(8.511456,-2.1241992){$\beta^+$}
\usefont{T1}{ptm}{m}{n}
\rput(8.611455,-3.8641992){$\beta^-$}
\usefont{T1}{ptm}{m}{n}
\rput(2.301455,-0.66419923){$\gamma^+$}
\usefont{T1}{ptm}{m}{n}
\rput(1.761455,-5.284199){$\gamma^-$}
\psbezier[linewidth=0.04,linecolor=color4115,arrowsize=0.05291667cm 2.0,arrowlength=1.4,arrowinset=0.4]{<-}(6.38,-3.049199)(7.92,-3.1291993)(7.88,-3.6891992)(8.24,-4.169199)
\psbezier[linewidth=0.04,linestyle=dashed,dash=0.16cm 0.16cm](0.64,-3.4291992)(1.62,-3.0091991)(3.6,-2.8091993)(4.6,-3.5691993)
\psbezier[linewidth=0.04,linestyle=dashed,dash=0.16cm 0.16cm](0.64,-2.6491992)(1.62,-3.0691993)(3.6,-3.2691991)(4.6,-2.5091991)
\psbezier[linewidth=0.04,linestyle=dashed,dash=0.16cm 0.16cm](3.0,-1.0691992)(2.58,-2.049199)(2.38,-4.029199)(3.14,-5.029199)
\usefont{T1}{ptm}{m}{n}
\rput(8.079726,-0.6441992){{\bf Case $m$ odd and $j$ even}}
\psline[linewidth=0.04cm](0.3,-5.7691994)(11.66,-5.7491994)
\psline[linewidth=0.04cm,linecolor=red,linestyle=dotted,dotsep=0.16cm](1.0,-9.409199)(4.52,-7.2291994)
\psline[linewidth=0.04cm,linecolor=red,linestyle=dotted,dotsep=0.16cm](2.64,-6.3691993)(2.66,-10.469199)
\psline[linewidth=0.04cm,linecolor=red,linestyle=dotted,dotsep=0.16cm](0.84,-7.3291993)(4.48,-9.409199)
\psbezier[linewidth=0.04,arrowsize=0.05291667cm 2.0,arrowlength=1.4,arrowinset=0.4]{<-}(4.3,-6.589199)(4.1,-7.469199)(2.16,-8.9492)(0.86,-9.089199)
\psbezier[linewidth=0.04,linestyle=dashed,dash=0.16cm 0.16cm](3.14,-10.329199)(2.78,-9.409199)(2.88,-9.849199)(2.68,-8.429199)
\psbezier[linewidth=0.04,linestyle=dashed,dash=0.16cm 0.16cm](2.64,-8.389199)(3.5180962,-9.011902)(3.593438,-8.975145)(4.08,-9.909199)
\psline[linewidth=0.04cm,linecolor=red,linestyle=dotted,dotsep=0.16cm](5.7,-8.589199)(9.8,-8.589199)
\psbezier[linewidth=0.04,linecolor=color4115,arrowsize=0.05291667cm 2.0,arrowlength=1.4,arrowinset=0.4]{->}(5.68,-7.669199)(7.06,-8.9492)(8.64,-8.889199)(9.7,-7.509199)
\usefont{T1}{ptm}{m}{n}
\rput(9.971455,-7.2641993){$\beta^+$}
\usefont{T1}{ptm}{m}{n}
\rput(5.591455,-7.444199){$\beta^-$}
\usefont{T1}{ptm}{m}{n}
\rput(4.421455,-6.304199){$\gamma^+$}
\usefont{T1}{ptm}{m}{n}
\rput(0.9014551,-8.704199){$\gamma^-$}
\usefont{T1}{ptm}{m}{n}
\rput(8.059727,-9.464199){{\bf Case $m$ even and $j$ odd}}
\end{pspicture} 
}
\caption{The cocritical set $f^{-1}(f(\g))=f^{-1}(\beta)$. We have kept the red lines for reference. In each sector $S$ bounded by red lines, the number of branches of $f^{-1}(f(\g))$ is equal to that of $F^{-1}(F(\g))$ (refer to Figure \ref{cases}), except in 
the two sectors containing $\g^\pm$, where $f^{-1}(f(\g))\subset \g$. }\label{co-critical}
\end{figure}

\subsection{Prescribing numerical invariants or   local models for harmonic mappings}

Now we are ready to prove Corollary \ref{existence}. Due to the equality $j+m^2=\mu$,  we only need to prove that   given two integers $\mu,m$ satisfying  $m\ge 1$ and $\mu\ge m^2+m$ there exist
harmonic maps  of the form $g(z)=p(z)^m-\bar z^m$   such that $\mu(g,0)=\mu$.

  Assume that $p(z)= z+ bz^2+ o(z^2)$ with $|b|=1$. 

In the case $\mu=m^2+m$, one can take $p$ such that $(-b^2)^m\ne 1$
and apply Lemma \ref{relation}. 

Assume now $\mu> m^2+m$, in particular $\mu> 2$. Choose $b=i$. Then $-b^2=1$ is always a $m$-th root of unity. 
And $f_{-b^2}(z)=f_1(z)=p(z)-\bar z$. Choose $p$ such that $\mu(f_1,0)-2=\mu-(m^2+m)$ and apply Lemma \ref{relation}. 

Now given a pair of positive integers $n^\pm$ with $n^+=n^-$, resp. with $|n^+-n^-|=1$, one can use the table \Ref{n-plus} to find a suitable pair $m$ and $j$,
or the table \Ref{N-plus} to find a suitable pair $m$ and $ \mu$, and proceed as above to find an harmonic map realising the model.  
\qed

Here are some concrete examples realizing a given pair $(\mu,m)$ with $\infty\ge \mu\ge m^2+m$.

If $\infty>\mu=m^2+m$, take any $p(z)=z+bz^2$ with $|b|=1$ and  $(-b^2)^m\ne 1$. Then $\mu(p(z)^m-\bar z^m,0)=\mu$ .

If $\infty>\mu> m^2+m$, set $\nu=\mu-(m^2+m)+2=\mu-(m-1)(m+2)$ and
$p_\nu(z)=z\sum_{s= 0}^{\nu-2} (iz)^s + a z^\nu$ with $\Re a\ne 0,\pm 1$
and $g_\nu(z)=(p_\nu(z))^m-\zbar^m$. Then $\mu(g,0)=\mu$.

If $\mu=\infty$, set  
$p(z)=-\dfrac{z}{1-z}$ and $g(z)=p(z)^m-\overline z^m$. We have $\overline p\circ p(z)=p\circ p(z)=z$, and
$$\mu(g,0)= \sum_{\xi^m=1,\eta^m=1, \xi,\eta\ne 1}Ord_0 \Big(\eta\, \overline p(\xi \, p(z))-z\Big)+ Ord_0 \Big( \overline p(p(z))-z\Big)= \infty.$$
One can also check by hand that  $j(g,0)=\infty$.

\appendix

\section{Analytic  planar maps at a regular critical point}

Let $K=\R$ or $\C$. Let  $f: \lv x \\ y \rv \mapsto f\lv x \\ y \rv $ be a $K$-analytic map with $a$ as a regular critical point. The critical set $\CCC_f$, as a level set of $J_f$, is everywhere orthogonal to the gradient  vector field $(\partial_x J_f, \partial _y J_f)$. The unique curve $\G(t)$ satisfying
$$\G'(t)= \lv -\partial_yJ_f(\G(t))\\ \partial_xJ_f(\G(t))\rv, \quad \G(0)=a\ $$  thus parametrizes the critical curve $\{J=0\}$. Now the map $f$ transports the curve $\G(t)$ to the critical value set,
inducing thus a natural local  parametrization $t\mapsto \Sigma(t)=f(\G(t))$.

In this section we prove that the critical value curve of a $K$-analytic map  at a regular critical point takes always a pair $(j, j+1)$ as its order-pair, and $\mu=j+1$.  
 We then give a recursive algorithm computing $j$, thus $\mu$. 
 
 \subsection{Critical value order-pair and multiplicity}
 
 \REFTHM{real-analytic} Let $K= \R$ or $\C$. Let $W\subset K^2$ be an open neighborhood of $w_0\in K^2$, and $F: W\to K^2$  a $K$-analytic mapping with  $w_0$ as a regular critical point. Then,
\begin{enumerate} \item (critical value order-pair) Let $j$ be the order at $F(w_0)$ of the critical value curve in its natural parametrization. This curve has an order-pair   of the form $(1,\infty)$ if $j=1$, and  $(j,j+1)$ if $1<j<\infty$.  
\item  (critical value order and multiplicity)  The order $j$ is related to the multiplicity by the formula
\REFEQN{formula-regular}   j+1=\mu(F,w_0).\ENDEQN
\item (topological model in the reals) In the case $K=\R$ and $\mu(F,w_0)<\infty$, \\
$\left\{\!\!\!\text{\begin{tabular}{l}  $\mu(F,w_0)$ even, or\\ $\mu(F,w_0)$ odd \end{tabular}}\right.$ iff 
 there is a pair of topological local changes of coordinates  $h,H$ of $\R^2$,
so that $H\circ F\circ h$ takes the standard  $\left\{\!\!\!\text{\begin{tabular}{l} fold form $\lv x\\ y \rv\mapsto \lv x \\ y^2\rv$, or \\
 cusp form $\lv x\\ y\rv\mapsto \lv x\\ xy+ y^3\rv$.\end{tabular}}\right.$ In particular, outside the critical value set, 
the number of preimages  is either 0 or 2 in the fold case, and 1 or 3 in the cusp case. 
\end{enumerate}\ENDTHM
\beginp We will make a sequence of analytic changes of coordinates to $F$. This will lead to 
new maps whose critical value curves differ from that of $F$ by analytic changes of coordinates. We will see that in some suitable coordinates
the critical value curve has an order pair in the form $(j, j+1)$.  If $j>1$ then $j, j+1$ are co-prime and the pair becomes then an analytic invariant.
It follows that the  critical value curve of our original map has also the same order-pair. 

Precompose $F$ by a translation if necessary we may assume ${ w_0}={\bf 0}$. Denote by $DF_{\bf 0}$ the
differential of $F$ at ${\bf 0}$.

Precompose and post-compose $F$ by some rotations if necessary we may assume that $Ker (DF_{\bf 0})$ is the $y$-axis and $Image(DF_{\bf 0})$ is the $x$-axis.
Divide  $F$ by a non-null constant in $K$ if necessary we may further assume $DF_{\bf 0}\lv 1\\ 0 \rv=\lv 1 \\ 0\rv$. 

It follows that the Jacobian matrix $Jac_F({\bf 0})$ is $\lv 1 & 0 \\ 0 & 0 \rv$.

Write now $F\lv x \\ y \rv=\lv f(x,y)\\ g(x,y)\rv $ and set $\phi\lv x \\ y \rv=\lv f(x,y)\\ y \rv$. Then $Jac_{\phi}({\bf 0})=Id$.
It follows that $\phi$ is a local diffeomorphism.

Replace now $F$ by $F\circ \phi^{-1}$ we may assume $F$ takes the form
$$ F\lv x \\ y \rv=\lv x \\ g(x,y)\rv, \quad DF_{\bf 0}=\lv 1 & 0 \\ 0 & 0 \rv \ .$$

By assumption $(\nabla J_F)_{\bf 0}\ne (0,0)$. Set $(\nabla J_F)_{\bf 0}=(b,b')$. Note that $J_F=\dfrac{\partial g}{\partial y}$. We get thus the following local expansion
$$\dfrac{\partial g}{\partial y} = 0 + b x + b' y + R_{\ge 2}(x,y) = x (b+a(x,y) )+ y^j(c + R_{\ge 1}(y))$$
for some $\infty\ge j\ge 1$ ($\infty\ge j\ge 2$ if $b'=0$) and $b\cdot c\ne 0$, where the function $a(x,y)$ has no constant term.

The case $b'\ne 0$. The curve $\beta$ can be parametrized by $x$ and has the order-pair    $(1,\infty)$ at $0$. \marginpar{details ??}

We will only treat the  case $b\cdot b'\ne 0$. Then $j\ge 2$. We will prove that the critical value 
curve has $(j, j+1)$ as its order-pair at ${\bf 0}$.

Set $B(x)=g(x,0)$. We have $B(0)=0$. Post-compose now $F$ by $\lv u \\ v\rv \mapsto \lv u \\ v-B(u)\rv$ we may assume \REFEQN{FF} F \lv x \\ y\rv=\lv x \\ g(x,y)\rv\quad \text{with}\quad  g(x,y)= \ds  \int_0^y \dfrac{\partial g}{\partial y}(x,y) dy. \ENDEQN

{\bf Case $j=\infty$}.  We see that $\CCC_F$ is locally the $y$-axis   and $g(x,y)=   b\,xy \cdot A(x,y)$ with $A(0,0) =1$.
So  \REFEQN{infinit} F\lv x\\ y\rv =\lv x \\ b\,xy\cdot A(x,y)\rv, \quad A(0,0)=1\ENDEQN
and  $F|_{\CCC_F}$ is locally  constant. We may also make a change of variable $\lv x_1 \\ y_1 \rv=\lv  x \\ b\,y\cdot A(x,y) \rv=\Phi\lv x\\ y\rv$. 
Clearly  $\Phi$ is locally invertible and $F\circ \Phi^{-1} \lv x_1 \\ y_1 \rv= \lv  x_1\\ x_1y_1\rv$.

{\bf Case  $j<\infty$}. The map $F$ takes the form
$$F\lv x\\ y\rv =\lv x \\  \dfrac{y^{j+1}}{j+1}(c + s(y))+b\,xy\cdot A(x,y)\rv, \quad A(0,0)=1.$$

Let $y_1$ be the analytic function in $y$ tangent to the identity at $0$  so that $y_1^{j+1}=y^{j+1}(1+s(y)/c)$
and change the variable $y$ to $y_1$, 
one can further reduce $F$ to the following form (by abuse of notation we use again $y$ to denote the new variable):
\REFEQN{finite } F\lv x\\ y\rv =\lv x \\  \dfrac{c\,y^{j+1}}{j+1} +b\,xy\cdot \hat A(x,y)\rv, \quad  \hat A(0,0)=1.\ENDEQN

Note that $J_F(x,y)=\dfrac{\partial}{\partial y} \left(\dfrac{c\,y^{j+1}}{j+1} +b\,xy\cdot \hat A(x,y) \right)=c\, y^j+ b\cdot x(1+C(x,y))$
for some function $C(x,y)$ that vanishes at $(0,0)$.

Solving now the implicit equation $ J_F(x,y)=0$, we see that the critical set $\CCC_F$ is locally parametrized  by $y\mapsto \g(y)=\lv  x(y)\\ y\rv$
with
$x(y)= -\dfrac c{b} y^j+R_{\ge j+1}(y)$.

We may now compute the critical value curve in this coordinate :
$$\beta(y)=F(\g(y))= F\lv x(y)\\ y\rv=\lv x(y) \\ \dfrac{c\,y^{j+1}}{j+1}-b\dfrac{c\,y^j}{b}y+ R_{\ge j+2} (y)\rv =\lv -\dfrac{c\,y^j} {b} +R_{\ge j+1}(y)\\ \dfrac{-c\,j}{j+1}y^{j+1} + R_{\ge j+2} (y)\rv \ .$$

It follows that $\beta$ has the order-pair    $(j,j+1)$ at $0$. This proves Point 1.

Let us prove that the  multiplicity $\mu(F,{\bf 0})$ of $F$ at ${\bf 0}$ is $j+1$. 
This multiplicity is equal to $$\limsup_{(x,y)\to {\bf 0}} \# \left(U\cap F_{\C}^{-1}\lv x \\ y \rv\right)$$
with $U$ a small neighborhood of ${\bf 0}$ in $\C^2$.

By \Ref{finite } if $x=0$ and $y$ is close to $0$ but $y\ne 0$, then $\# F^{-1}\lv 0 \\ y \rv =j+1$. This is also true for $\lv x \\ y \rv$ close to ${\bf 0}$ by Rouch\'e's theorem applied to the second coordinate function of $F$ as a function of $y$. We know that 
$$\forall\, \ep >0, \ \exists\,\eta>0\ s.t.\ \forall |s|,|t|<\eta,\ |g(s,y)-t|_{|y|=\ep}\ne 0\ .$$
This proves Point 2.

Point 3. Assume $K=\R$. All   functions below will have real coefficients. One can write $x(y)=-\dfrac c b (y+R_{\ge 2} (y))^j$.

If $j+1$ is even the function $x(y)$ is a local homeomorphism of an interval about $0$ so has a unique inverse. It follows that on each vertical line $x=c$ for $c$
small, the map $F$ has a unique critical point. Since $F$ sends the line into itself, it must be a fold.

The case $j+1$ odd: $x(y)$ is a convex curve staying on one half plane, say the left half plane. It follows that for  every $c>0$ small, the map
$F$ sends the  vertical line $x=c$  homeomorphically to itself. And for $c<0$ small $F$ on the line $x=c$ behaves topologically as $ay+ y^3$ with $a<0$.
So $F$ is a topological cusp.  

The rigorous constructions of the changes of coordinates are very similar to Claims 1 and 8 in the proof of Theorem \ref{Lyzzaik2}. As our map $F$ here preserves vertical lines, we may instead choose to pull back small rectangles $]-r,r[\times ]-s,s[$ so that the upper and lower boundary segments
do not intersect the critical and co-critical sets. We omit the details. 
\qed

Remark that in the proof we have also established a collapsing model in $K$: We have  $\mu(F,w_0)=\infty$ iff  there is a pair of $K$-analytic changes of coordinates $\varphi$ and $\Phi$
so that $\Phi\circ F\circ \varphi$ 
takes the standard collapsing form  $\lv x\\ y\rv \mapsto \lv x \\ xy\rv$.

Remark also that Whitney has given a geometric model in $K$ for the 'stable singularity' cases: $\left\{\!\!\!\text{\begin{tabular}{l} $\mu(F,w_0)= 2$\\  $\mu(F,w_0)=3$\end{tabular}}\right.$
iff 
there is a pair of $K$-analytic changes of coordinates $h,H$
so that $H\circ F\circ h$ takes the standard $\left\{\!\!\!\text{\begin{tabular}{l} fold form $\lv x\\ y \rv\mapsto \lv x \\ y^2\rv$\\
 cusp form $\lv x\\ y\rv\mapsto \lv x\\ xy+ y^3\rv$.\end{tabular}}\right.$

\subsection{A recursive algorithm computing $j$}

 This subsection is inspired by a conversation with H.H. Rugh.

For  a $C^\infty$ planar mapping $f$, let $J$ be the jacobien of $f$. In the following both our
 domaine and range planes will be $\R^2$ identified with $\C$. In this spirit the jacobien will
also be considered as a map with range in $\R$.

 Consider now a map $f$ from $U$ to $\C$, we define 
$$\nabla_{\R} f=(f_x, f_y)\text{ and  }\nabla f=(f_z,f_{\bar z}):=\Big(\dfrac12(f_x-if_y),\dfrac12(f_x+if_y)\Big).$$
 Mimicking Whitney's definition for folds and cusps, we set recursively
\REFEQN{M}M_1=\left|\begin{array}{ll} \nabla_{\R} J \\ \nabla_{\R} f\end{array}\right|,\ M_2=
\left|\begin{array}{ll} \nabla_{\R} J \\ \nabla_{\R} M_1\end{array}\right|, \cdots,\ 
M_k=\left|\begin{array}{ll} \nabla_{\R} J \\ \nabla_{\R} M_{k-1}\end{array}\right|,\cdots\ ;\ENDEQN
\REFEQN{L}L_1=\left|\begin{array}{ll} \nabla f \\ \nabla J\end{array}\right|,\ L_2=
\left|\begin{array}{ll} \nabla L_1 \\ \nabla J\end{array}\right|, \cdots,\ 
L_k=\left|\begin{array}{ll} \nabla L_{k-1} \\ \nabla J\end{array}\right|,\cdots\ .\ENDEQN

\REFPROP{Reals}  Let $f:(\R^2,a)\to (\C,f(a))$ be a smooth map. We have, 
\[ \forall\ n\ge 1, M_n=(2i)^nL_n.
\]
Let $\G(t)$ the trajectory of the vector field $(-J_y(z), J_x(z))$ with initial point $a$, and set $\Sigma(t)=f(\G(t))$. We have
\[ \forall\ n\ge 1,\ \Sigma^{(n)}(t)= M_n(\G(t))=(2i)^nL_n(t).
\]
In particular 
$ \Sigma^{(n)}(0)=M_n(a)=(2i)^nL_n(a)\ .$
\ENDPROP
\beginp 
Let $G,H: (\R^2,a)\to (\C,G(a))$ be two $C^\infty$ smooth mappings. 

I. We claim first
\REFEQN{generals} \left|\begin{array}{ll} \nabla_{\R} H \\ \nabla_{\R}G \end{array}\right|=2i \left|\begin{array}{ll} \nabla G \\ \nabla H \end{array}\right|.\ENDEQN
Proof.  
Recall that $\nabla_{\R}H=(H_x,H_y)$ and $\nabla_{\R}G=( G_x, G_y)$.
It follows from 
$G_z=\dfrac12 (G_x-iG_y)$ and $G_{\overline z} =\dfrac 12 (G_x+iG_y)$  that $G_x=G_z+ G_{\overline z}$ and  $G_y=i(G_z- G_{\overline z})$.
So $$ \left|\begin{array}{ll} \nabla_{\R} H \\ \nabla_{\R}G \end{array}\right|=
 \left|\begin{array}{ll} H_x & H_y  \\ G_x & G_y  \end{array}\right| =
i \left|\begin{array}{ll} H_z+ H_{\zbar} & H_z-H_{\zbar} \\ G_z+G_{\zbar} & G_z-G_{\zbar} \end{array}\right| =-2i \left|\begin{array}{ll} H_z & H_{\zbar} \\ G_z & G_{\zbar} \end{array}\right| =-2i \left|\begin{array}{ll} \nabla H \\ \nabla G \end{array}\right| .$$

II. Apply now \Ref{generals} to $G=f$ and $H=J$, we get $M_1=(2i)L_1$, then to $G=M_1$ we get
$$M_2=\left|\begin{array}{ll} \nabla_{\R} J \\ \nabla_{\R}M_1 \end{array}\right|
=(2i) \left|\begin{array}{ll} \nabla_{\R} J \\ \nabla_{\R}L_1 \end{array}\right|\stackrel{\Ref{generals}}{=}(2i)^2\left|\begin{array}{ll} \nabla L_1 \\ \nabla J \end{array}\right|= (2i)^2 L_2\ .$$
By induction \REFEQN{all} M_n=(2i)^n L_n,\quad \forall\,n\ge 1\ .\ENDEQN

III. We claim now
\REFEQN{Generals} \dfrac d {dt} G(\G(t))=\left|\begin{array}{ll} \nabla_{\R} J \\ \nabla_{\R}G \end{array}\right|_{\G(t)}\ .\ENDEQN
Proof. Using the fact that $\G'(t)=(-J_y, J_x)|_{\G(t)}$,  for any $v\in \C^2$ we have $$  \langle v, \G'(t)\rangle= \left|\begin{array}{c} J_x \  J_y \\ v \end{array}\right|_{\G(t)}=\left|\begin{array}{c} \nabla_{\R} J \\ v \end{array}\right|_{\G(t)}.$$
 Write $ G: \lv x\\ y\rv\mapsto \lv G_1(x,y)\\ G_2(x,y)\rv$ (by identifying the range plane to $\R^2$). 
 Then $$\dfrac d {dt} G(\G(t))=DG|_{\G(t)}(\G'(t))=\lv \mystrut \langle
 \nabla_{\R}G_1(\G(t)),\G'(t)\rangle \\ \langle \nabla_{\R}G_2(\G(t)),\G'(t)\rangle\rv=
\lv \mystrut \left|\begin{array}{ll} \nabla_{\R} J \\ \nabla_{\R}G_1 \end{array}\right| \vspace{0.2cm}\\
\left|\begin{array}{ll} \nabla_{\R}J \\ \nabla_{\R}G_2\end{array}\right|\rv_{\G(t)} \ . $$
Identify $G(x,y)$ with $G_1(x,y)+i G_2(x,y)$.  We have
$$\dfrac d {dt} G(\G(t))=  \left|\begin{array}{ll} \nabla_{\R} J \\ \nabla_{\R}G_1 \end{array}\right|_{\G(t)}+i\left|\begin{array}{ll} \nabla_{\R}J \\ \nabla_{\R}G_2\end{array}\right|_{\G(t)}= \left|\begin{array}{ll} \nabla_{\R} J \\ \nabla_{\R}G \end{array}\right|_{\G(t)}\ .$$

IV. Apply now \Ref{Generals} inductively to $G=f,M_1,M_2,\cdots$, we get
$$\Sigma'(t)=\dfrac{d}{dt} f(\G(t))= \left|\begin{array}{ll} 
                       \nabla_{\R} J \\ \nabla_{\R} f                                     \end{array}\right|_{\G(t)}=M_1(\G(t));$$
 $$\Sigma''(t)= \dfrac{d}{dt} M_1(\G(t))=\left|\begin{array}{ll} 
                       \nabla_{\R} J \\ \nabla_{\R} M_1                                     \end{array}\right|_{\G(t)}= M_2(\G(t))\ .$$
By induction $\Sigma^{(n)}(t)=M_n(\G(t))$, and $\Sigma^{(n)}(0)=M_n(\G(0))= M_n(a)$.

Combining with \Ref{all} we get $\Sigma^{(n)}(0)=(2i)^n L_n(a)$ as well.
\qed
\begin{corollary}
The invariant $j$ is the first integer $n$ for which $L_n(0)\ne 0$.
\end{corollary}

\section{Examples}

These examples illustrate some differences 
between the harmonic and the general real analytic case. 

1. General remarks: 

a. In the harmonic case, by a theorem of Hans Lewy, 
the  locus of non local injectivity is the same as the 
critical set. In particular, this implies that $C_f$ 
and hence $V_f$  have a topological meaning. 
This is no longer true in the real analytic case, 
as is shown for instance by the map: 
$(x,y)\to (x,y^3)$. 

b. One can easily  check  that 
for a real analytic  planar germ $g$ 
from $({\Bbb C},0)$ into itself, 
the critical set  (and the locus of non injectivity )
are the same for the germs  $g$ and $g^n$ ($n\in {\Bbb N}^*)$ outside the origin.  

c. As  the case of a regular critical point was studied before, we 
give examples for which the gradient of the Jacobian vanishes at the origin. 

 Here are examples of planar analytic germs $g$  at the origin ($C_g$ is smooth and coincides with 
the locus of non local injectivity).

A. In general, the condition $\mu=j+m^2$ is not satisfied in the analytic case. 
The map $g(x,y)= (x, x^2y^2+y^4)$ is a simple example. The critical set 
is the $x$-axis.  One has: $j=m=1$ and $\mu=4$.

B. Topological differences. 

1. $g(x,y)=  (x+iy^2)^2$  has local topological degree $0$ at the origin,
but $V_f$
is $[0,+\epsilon [$.  So the germ at the origin is not topologically equivalent 
to the germ of a harmonic map. 

2.  A real analytic germ with a smooth critical set  can have any local  topological degree. For instance, take 
a harmonic map $f$ of degree $1$  or $-1$ and put  $g= f^n$. This map 
has  local topological degree $n$ or $-n$, and then,  is not   topologically 
equivalent to a harmonic germ for $n>1$. 

C. Problem for the parametrization of $V_g$. 

$g(x,y)=(x+iy^2)^3$.  The parametrization of $V_g$  which comes 
from  the parametrization $x=t, y=0$ of  $C_f$  is of the 
non-injective  form $x=t^3, y= 0$, when $t$ is complex. 

D. Examples with Puiseux pair   $(2,5)$ 

$g(x,y)= (x^2+y^2, x^5 +cx^3y^2+xy^4)$. 

The Jacobian is equal to $2y((2c-5)x^4+(4-3c)x^2y^2-y^4)$. Then 
For $4/3<c<5/2$, one gets an example satisfying the wanted conditions ($C_g=\{y=0\}$).
Moreover, it is topologically a fold. 

One can also check that $m=j=2$ , $\mu=\infty$  for $c= 2$  and $\mu=10$ otherwise.

LAREMA, D\'epartement de Math\'ematiques, Universit\'e d'Angers, 2, Bd Lavoisier, 49045 Angers, Cedex 01, France.
e-mail addresses:
\\
Mohammed.ElAmrani@univ-angers.fr\\
Michel.Granger@univ-angers.fr\\
Jean-Jacques.Loeb@univ-angers.fr\\
tanlei@math.univ-angers.fr

\end{document}